\documentclass[11pt,amsfonts]{amsart}

\usepackage{amsmath,amssymb,amsthm,enumerate}
\usepackage[all]{xy}
\usepackage{graphicx}
\usepackage{xcolor}
\usepackage{comment}
\usepackage[T1]{fontenc}
\usepackage{graphicx}
\usepackage{calc}

\SelectTips{eu}{12}

\newtheorem{thm}{Theorem}[section]
\newtheorem{prop}[thm]{Proposition}
\newtheorem{lem}[thm]{Lemma}
\newtheorem{cor}[thm]{Corollary}

\newtheorem{problem}[thm]{Problem}

\theoremstyle{definition}
\newtheorem{definition}[thm]{Definition}
\newtheorem{example}[thm]{Example}

\theoremstyle{remark}
\newtheorem{remark}[thm]{Remark}

\numberwithin{equation}{section}

\newcommand{\QQ}{\mathbb{Q}}
\newcommand{\RR}{\mathbb{R}}
\newcommand{\ZZ}{\mathbb{Z}}
\newcommand{\NN}{\mathbb{N}}

\DeclareMathOperator{\cl}{\mathrm{cl}}
\DeclareMathOperator{\scl}{\mathrm{scl}}

\newcommand{\bG}{N}
\newcommand{\hG}{G}

\newcommand{\CC}{\mathcal{C}}
\newcommand{\Surf}{\mathrm{Surf}}

\newcommand{\CSD}{\mathcal{C}_{\mathrm{Surf}_D}}

\newcommand{\Ker}{\mathrm{Ker}}

\newcommand{\Symp}{\mathrm{Symp}}
\newcommand{\Ham}{\mathrm{Ham}}

\newcommand{\Cal}{\mathrm{Cal}}

\newcommand{\HHH}{\mathrm{H}}

\newcommand{\GL}{\mathrm{GL}}

\newcommand{\Sp}{\mathrm{Sp}}
\newcommand{\Mod}{\mathrm{Mod}}

\newcommand{\ab}{\mathrm{ab}}
\newcommand{\Ab}{\mathrm{Ab}}
\newcommand{\OR}{\mathrm{OR}}
\newcommand{\supp}{\mathrm{supp}}

\newcommand{\rk}{\mathrm{rk}}
\newcommand{\intrk}{\mathrm{int}\textrm{-}\mathrm{rk}}
\newcommand{\sperk}{\mathrm{spe}\textrm{-}\mathrm{rk}}
\newcommand{\genrk}{\mathrm{gen}\textrm{-}\mathrm{rk}}

\newcommand{\bigast}{
  \mathop{
    \vphantom{\bigoplus}
    \mathchoice
      {\vcenter{\hbox{\resizebox{\widthof{$\displaystyle\bigoplus$}}{!}{$\ast$}}}}
      {\vcenter{\hbox{\resizebox{\widthof{$\bigoplus$}}{!}{$\ast$}}}}
      {\vcenter{\hbox{\resizebox{\widthof{$\scriptstyle\oplus$}}{!}{$\ast$}}}}
      {\vcenter{\hbox{\resizebox{\widthof{$\scriptscriptstyle\oplus$}}{!}{$\ast$}}}}
  }\displaylimits
}

\makeatletter
\newsavebox{\@brx}
\newcommand{\llangle}[1][]{\savebox{\@brx}{\(\m@th{#1\langle}\)}%
  \mathopen{\copy\@brx\kern-0.5\wd\@brx\usebox{\@brx}}}
\newcommand{\rrangle}[1][]{\savebox{\@brx}{\(\m@th{#1\rangle}\)}%
  \mathclose{\copy\@brx\kern-0.5\wd\@brx\usebox{\@brx}}}
\makeatother

\allowdisplaybreaks[1]

\begin{document}

\title{Mixed commutator lengths, wreath products and general ranks}

\author[M. Kawasaki]{Morimichi Kawasaki}
\address[Morimichi Kawasaki]{Department of Mathematical Sciences, Aoyama Gakuin University, 5-10-1 Fuchinobe, Chuo-ku, Sagamihara-shi, Kanagawa, 252-5258, Japan}
\email{kawasaki@math.aoyama.ac.jp}

\author[M. Kimura]{Mitsuaki Kimura}
\address[Mitsuaki Kimura]{Department of Mathematics, Kyoto University, Kitashirakawa Oiwake-cho, Sakyo-ku, Kyoto 606-8502, Japan}
\email{mkimura@math.kyoto-u.ac.jp}

\author[S. Maruyama]{Shuhei Maruyama}
\address[Shuhei Maruyama]{Graduate School of Mathematics, Nagoya University, Furocho, Chikusaku, Nagoya, 464-8602, Japan}
\email{m17037h@math.nagoya-u.ac.jp}

\author[T. Matsushita]{Takahiro Matsushita}
\address[Takahiro Matsushita]{Department of Mathematical Sciences, University of the Ryukyus, Nishihara-cho, Okinawa 903-0213, Japan}
\email{mtst@sci.u-ryukyu.ac.jp}

\author[M. Mimura]{Masato Mimura}
\address[Masato Mimura]{Mathematical Institute, Tohoku University, 6-3, Aramaki Aza-Aoba, Aoba-ku, Sendai 9808578, Japan}
\email{m.masato.mimura.m@tohoku.ac.jp}

\makeatletter
\@namedef{subjclassname@2020}{%
\textup{2020} Mathematics Subject Classification}
\makeatother


\begin{abstract}
  In the present paper, for a pair $(G,N)$ of a group $G$ and its normal subgroup $N$, we consider the mixed commutator length $\mathrm{cl}_{G,N}$ on the mixed commutator subgroup $[G,N]$. We focus on the setting of wreath products: $ (G,N)=(\mathbb{Z}\wr \Gamma, \bigoplus_{\Gamma}\mathbb{Z})$. Then we determine mixed commutator lengths in terms of  the general rank in the sense of Malcev. As a byproduct, when an abelian group $\Gamma$ is not locally cyclic, the ordinary commutator length $\mathrm{cl}_G$ does not coincide with $\mathrm{cl}_{G,N}$ on $[G,N]$  for the above pair. On the other hand, we prove that if $\Gamma$ is locally cyclic, then for every pair  $(G,N)$ such that $1\to N\to G\to \Gamma \to 1$ is exact, $\mathrm{cl}_{G}$ and $\mathrm{cl}_{G,N}$ coincide on $[G,N]$. We also study the case of permutational wreath products when the group $\Gamma$ belongs to a certain class related to surface groups.
\end{abstract}

\maketitle





\section{Introduction}
\subsection{Background}
Let $G$ be a group and $N$ be a normal subgroup of $G$. An element of the form $[g,x] = gxg^{-1}x^{-1}$ with $g\in G$ and $x\in N$ is called a \emph{$(G,N)$-commutator} or a  \emph{mixed commutator}; $[G,N]$ is the subgroup generated by mixed commutators, and it is called the \emph{$(G,N)$-commutator subgroup} or the \emph{mixed commutator subgroup}. The \emph{$(G,N)$-commutator length} or the \emph{mixed commutator length} $\cl_{G,N}$ is the word length on $[G,N]$ with respect to the set of mixed commutators. Namely, for $x\in [G,N]$, $\cl_{G,N}(x)$ is the smallest integer $n$ such that there exist $n$ mixed commutators whose product is $x$. In the case of $N=G$, the notions of $[G,N]$ and $\cl_{G,N}$ coincide with those of the \emph{commutator subgroup} $[G,G]$ and the \emph{commutator length} $\cl_G$, respectively. Commutator lengths have been extensively studied in geometric topology (for example, see \cite{EK}, \cite{BIP}, \cite{Tsuboi12} and \cite{Tsuboi13}).

It is a classical fact that $\cl_G(y)$ has the following geometric interpretation. We can regard an element $y$ in $[G,G]$ as a homotopy class of a loop of the classifying space $BG$ of $G$. Since $y \in [G,G]$, the loop can be extended as a continuous map from a compact connected surface $\Sigma$ with $\partial \Sigma = S^1$ to $BG$; the commutator length coincides with the minimum genus of such $\Sigma$.
This description can be generalized to the setting of the mixed commutator length $\cl_{G,N}$: in \cite[Theorem~1.3]{KKMM1}, the authors obtained geometric and combinatorial interpretation of $\cl_{G,N}$, which is explained in terms of the concept of $(\hG,\bG)$-simplicial surfaces.

 In general, it is difficult to determine the precise values of $\cl_G$ and $\cl_{G,N}$. In the present paper, we investigate the comparison between mixed commutator lengths $\cl_{G,N}$ and usual commutator lengths $\cl_G$. Especially in the case of certain wreath products, we determine the precise value of $\cl_{G,N}$ in terms of general ranks due to Malcev \cite{Mal}.

In the present paper, we use the notation $\Gamma$ and $q$ to fit the following short exact sequence:
\begin{align*}
1\longrightarrow N\longrightarrow G\stackrel{q}{\longrightarrow} \Gamma \longrightarrow 1. \tag{$\star$}
\end{align*}

\subsection{Mixed commutator length on wreath products}

Our main result determines the mixed commutator length for certain elements in  wreath products. For a group $H$, $e_H$ denotes the group unit of $H$. For two groups $H$ and $\Gamma$, the (restricted) \emph{wreath product} $H\wr \Gamma$ is defined as  the semidirect product $\left(\bigoplus_{\Gamma}H\right)\rtimes \Gamma$, where $\Gamma$ acts on $\bigoplus_{\Gamma}H$ by shifting the coordinates.  By letting $G = H \wr \Gamma$, $G$ admits a natural projection $q \colon G \twoheadrightarrow \Gamma$ fitting in short exact sequence ($\star$). Note that in this case short exact sequence $(\star)$ splits.

We regard an element of $\bigoplus_{\Gamma} H$ as a function $u$ from $\Gamma$ to $H$ such that $u(\gamma) = e_H$ except for finitely many $\gamma \in \Gamma$. 
 For $\gamma \in \Gamma$, we write $u(\gamma) \in H$ to be the $\gamma$-entry of $u$.
 In the case where $H=\ZZ$, for $\lambda\in \Gamma$, $\delta_\lambda \colon \Gamma \to \ZZ$ denotes   the Kronecker delta function at $\lambda$, meaning that $\delta_\lambda(\gamma)=1$ if $\gamma = \lambda$ and $\delta_\lambda(\gamma)=0$ otherwise.
 For a group $\Gamma$ and a subset $S$, $\langle S\rangle$ denotes the subgroup of $\Gamma$ generated by $S$. In the present paper, set $\NN=\{1,2,3,\ldots\}$.

 Now we state our first result:

\begin{thm}\label{thm:wreath_new}
Let $\Gamma$ be a group. Set $G=\ZZ\wr \Gamma$ and $N=\bigoplus_{\Gamma}\ZZ$. Let $k\in \NN$ and $\lambda_1,\ldots ,\lambda_k\in \Gamma\setminus \{e_{\Gamma}\}$. Let $\Lambda=\langle \lambda_1,\ldots ,\lambda_k\rangle$. Set
\[
 x_{(\lambda_1,\ldots ,\lambda_k)}=\sum_{i=1}^k\delta_{\lambda_i}-k\delta_{e_{\Gamma}}.
\]
Then we have
\[
\cl_{G,N}( x_{(\lambda_1,\ldots ,\lambda_k)})=\intrk^{\Gamma}(\Lambda),
\]
where $\intrk^{\Gamma}(\Lambda)$ is the \emph{intermediate rank} of the group pair $(\Gamma,\Lambda)$, defined in Definition~$\ref{def:int_rk}$.
\end{thm}

The formula in Theorem~\ref{thm:wreath_new} is stated in terms of a variant of ranks of groups, which we call the \emph{intermediate rank}, as follows.  This notion relates to the concept of \emph{general rank} in the sense of Malcev \cite{Mal}. Here, $\ZZ_{\geq 0}:=\{n\in \ZZ\;|\; n\geq 0\}$.

\begin{definition}[intermediate rank]\label{def:int_rk}
\begin{enumerate}[(1)]
  \item For a group $\Lambda$, the \emph{rank} of $\Lambda$ is defined by
\[
\rk(\Lambda)=\inf \{ \#S \; | \; \Lambda = \langle S \rangle \} \ \in \ZZ_{\geq 0}\cup \{\infty\}.
\]
 \item For a pair  $(\Gamma,\Lambda)$ of a group $\Gamma$ and its subgroup $\Lambda$, the \emph{intermediate rank} of $(\Gamma,\Lambda)$ is defined by
\[
\intrk^{\Gamma}(\Lambda)=\inf \{ \rk(\Theta) \; | \; \Lambda \leqslant \Theta\leqslant \Gamma \}\  \in \ZZ_{\geq 0}\cup \{\infty\}.
\]
 \item (general rank, \cite{Mal}) For a group $\Gamma$, the \emph{general rank} of $\Gamma$ is defined by
\[
\genrk(\Gamma)=\sup\{\intrk^{\Gamma}(\Lambda)\;|\; \textrm{$\Lambda$ is a finitely generated subgroup of $\Gamma$}\}.
\]
\end{enumerate}
\end{definition}

 Furthermore, we have our main result, Theorem~\ref{thm:wreath_shin}, which generalizes Theorem~\ref{thm:wreath_new}. To state the theorem, we employ the following terminology.
\begin{definition}\label{def:support}
Let $T$ be a non-empty set and $A$ be an additive group. Let $u\in \bigoplus_{T}A$. We regard $u$ as a map $u\colon T\to A$ such that $u(t)=0$ for all but finitely many $t\in T$.
\begin{enumerate}[$(1)$]
  \item The \emph{support} $\supp(u)$ is defined as the set
\[
\supp(u)=\{t\in T \;|\; u(t)\ne 0\}.
\]
  \item For a subset $S$ of $T$, we say that $S$ is a \emph{zero-sum set for $u$} if
\[
\sum_{s\in S}u(s)=0.
\]
\end{enumerate}
\end{definition}

\begin{thm}[mixed commutator length on wreath products]\label{thm:wreath_shin}
Let $\Gamma$ be a group. Set $G=\ZZ\wr \Gamma$ and $N=\bigoplus_{\Gamma}\ZZ$.
\begin{enumerate}[$(1)$]
\item Let $\Lambda$ be a subgroup of $\Gamma$. Assume that $x\in N$ fulfills the following three conditions.
\begin{enumerate}[$(i)$]
  \item $e_{\Gamma}\in \supp(x)$.
  \item $\Lambda$ is a zero-sum set for $x$, namely, $\sum_{\lambda\in \Lambda}x(\lambda)=0$.
  \item For every zero-sum set $S$ for $x$ satisfying $S\subseteq \supp(x)$ and $e_{\Gamma}\in S$, we have $\langle S\rangle =\Lambda$.
\end{enumerate}
Then we have
\[
\cl_{G,N}(x)=\intrk^{\Gamma}(\Lambda).
\]
\item We have
\[
\{\cl_{G,N}(x)\;|\; x\in [G,N]\}=\{r\in \ZZ_{\geq 0}\;|\;r\leq \genrk(\Gamma)\}.\]
Furthermore, the following holds  true:  for every $r\in \NN$ with $r\leq \genrk(\Gamma)$, there exists $x_r\in [G,N]$ such that $\cl_{G,N}(x_r)=r$ and that $x_r$ fulfills conditions $(i)$--$(iii)$ of $(1)$ with $\Lambda=\langle \supp(x)\rangle$.
\end{enumerate}
\end{thm}
We note that $\Lambda$ in Theorem~\ref{thm:wreath_shin}~(1) must be finitely generated because $\supp(x)$ is a finite set.

On $\cl_G$, in \cite[Theorem~7.1]{KKMM1} we showed that if $(\star)$ splits, then
\begin{eqnarray}\label{eq:3bai}
\cl_{G,N}(x)\leq 3\cl_{G}(x)
\end{eqnarray}
holds for every $x\in [G,N]$. In this paper, we improve this bound for the case where $N$ is \emph{abelian} in the following manner.

\begin{thm}\label{thm:split}
Let $(G,N,\Gamma)$ be a triple of groups that fits in the short exact sequence $(\star)$. Assume that $(\star)$ splits and that $N$ is \emph{abelian}. Then, we have
\begin{eqnarray}\label{eq:2bai}
\cl_{G,N}(x)\leq 2\cl_{G}(x)
\end{eqnarray}
for every $x\in [G,N]$.
\end{thm}

Results in this paper show that \eqref{eq:2bai} is \emph{sharp}; see Theorem~\ref{thm:abelian}, Theorem~\ref{thm:surface_new} and Proposition~\ref{prop:cw}. We also note that  Theorem~\ref{thm:split} in particular applies to the triple $(G,N,\Gamma)=(\ZZ\wr \Gamma, \bigoplus_{\Gamma}\ZZ, \Gamma)$ for every group $\Gamma$.

\subsection{Coincidence problem of $\cl_G$ and $\cl_{G,N}$}
In the case that $\Gamma$ is abelian, it is easy to compute $\intrk^{\Gamma}(\Lambda)$ and $\genrk(\Gamma)$:
for an abelian group $\Gamma$, the intermediate rank coincides with $\rk(\Lambda)$, and the general rank $\genrk(\Gamma)$ coincides with the \emph{special rank of $\Gamma$}, which is defined as
the supremum of ranks of all finitely generated subgroups; see Definition~\ref{def:local_rank} and Lemma~\ref{lem:abel}.  This allows us to construct examples of pairs $(G,N)$ such that $\cl_G$ and $\cl_{G,N}$ are different.

Before discussing the coincidence problem of $\cl_G$ and $\cl_{G,N}$, we recall from our previous paper \cite{K2M3} that the stabilizations of $\cl_G$ and $\cl_{G,N}$ coincide in several cases. For $y \in [G,G]$ and $x \in [G,N]$, set
\[ \scl_G(y) = \lim_{n \to \infty} \frac{\cl_G(y^n)}{n} \quad \textrm{and} \quad \scl_{G,N}(x) = \lim_{n \to \infty} \frac{\cl_{G,N}(x^n)}{n}.\]
We call $\scl_G(y)$ the \emph{stable commutator length of $y$}, and $\scl_{G,N}(x)$ the \emph{stable mixed commutator length of $x$}. For a comprehensive introduction to the stable commutator lengths, we refer to Calegari's book \cite{Ca}.
By a celebrated result by Bavard (see \cite{Bav} or \cite{Ca}), called the Bavard duality theorem, it is well known that the stable commutator lengths are closely related to the notion of quasimorphisms of groups. In \cite{KKMM1} the authors established the Bavard duality theorem for $\scl_{G,N}$, which implies that $\scl_{G,N}$ are closely related to $G$-invariant quasimorphisms on $N$. Using this, in our previous work, we show the following coincidence result of $\scl_G$ and $\scl_{G,N}$ (see \cite[Proposition~1.6]{KKMM1} and \cite[Theorems~1.9, 1.10 and 2.1]{K2M3}).

\begin{thm}[\cite{KKMM1}, \cite{K2M3}]\label{thm:scl}
In $(\star)$, if $\Gamma$ is solvable and if either of the following conditions
\begin{enumerate}[$(1)$]
  \item short exact sequence $(\star)$ \emph{virtually splits}, meaning that, there exists $(\Lambda,s_{\Lambda})$ such that $\Lambda$ is a subgroup of finite index of $\Gamma$ and $s_{\Lambda}\colon \Lambda\to G$  a homomorphism satisfying $q\circ s_{\Lambda}=\mathrm{id}_{\Lambda}$;
  \item $\HHH^2(G;\RR)=0$;
  \item or, $\HHH^2(\Gamma;\RR)=0$
\end{enumerate}
is satisfied, then $\scl_{G}$ coincides with $\scl_{G,N}$ on $[G,N]$.
\end{thm}

Examples of pairs $(G,N)$ such that $\scl_G$ and $\scl_{G,N}$ do not coincide are also known; see \cite{KK} and \cite{MMM}.

Now we consider the coincidence problem of $\cl_G$ and $\cl_{G,N}$. There is no guarantee that $\cl_G$ and $\cl_{G,N}$ coincide even if $\scl_G$ and $\scl_{G,N}$ coincide. Nevertheless, $\cl_G$ and $\cl_{G,N}$ actually coincide in the following case. Here, a group is said to be \emph{locally cyclic} if every finitely generated subgroup is cyclic:


\begin{thm} \label{thm local cyclicity}
For every triple $(G,N,\Gamma)$ fitting in $(\star)$ such that $\Gamma$ is locally cyclic, $\cl_G$ and $\cl_{G,N}$ coincide on $[G,N]$.
\end{thm}

Cyclic groups, $\QQ$ and $\ZZ[1/2]/\ZZ$ are typical examples of locally cyclic groups, and $(\ZZ/2\ZZ)^2$ and $\RR$ are not locally cyclic. Every locally cyclic group is abelian. Note that $\Gamma$ is locally cyclic if and only if the general rank of $\Gamma$ is at most $1$. Hence Theorem~\ref{thm local cyclicity} can be rephrased as follows: if $\genrk (\Gamma) \le 1$, then $\cl_G$ and $\cl_{G,N}$ coincide on $[G,N]$.


Contrastingly, as an application of Theorems \ref{thm:wreath_new} and \ref{thm:wreath_shin}, we construct a pair $(G,N)$ such that $\cl_G$ and $\cl_{G,N}$ do not coincide when $\Gamma$ is abelian but not locally cyclic:

\begin{thm}[result for abelian $\Gamma$]\label{thm:abelian}
Let $\Gamma$ be an abelian group. Set $G=\ZZ\wr \Gamma$ and $N=\bigoplus_{\Gamma}\ZZ$. Then we have
\[
\{(\cl_G(x),\cl_{G,N}(x))\;|\;x\in [G,N]\}=\left\{ \Big( \left\lceil \frac{r}{2}\right\rceil,r \Big)\;|\; r\in \ZZ_{\geq 0},\ r\leq \sperk(\Gamma)\right\},
\]
where $\sperk(\Gamma)$ is the special rank of $\Gamma$, defined in Definition~$\ref{def:local_rank}$.
Here, $\lceil \cdot \rceil$ is the ceiling function.

In particular, if $\Gamma$ is not locally cyclic, then for the pair $(G,N)=(\ZZ\wr \Gamma,\bigoplus_{\Gamma}\ZZ)$, which fits in \emph{split} exact sequence $(\star)$, $\cl_{G}$ and $\cl_{G,N}$ do \emph{not} coincide on $[G,N]$. If moreover $\sperk(\Gamma)=\infty$, then for the same pair $(G,N)$,  we have
\[
\sup_{x\in [G,N]}(\cl_{G,N}(x)-C\cdot \cl_G(x))=\infty
\]
for every real number $C<2$.
\end{thm}

In particular, Theorem~\ref{thm:abelian} implies that Theorem~\ref{thm:scl} is no longer true if $\scl_G$ and $\scl_{G,N}$ are replaced with $\cl_G$ and $\cl_{G,N}$, respectively. This means that the coincidence problem of $\cl_G$ and $\cl_{G,N}$ is more subtle than the one of $\scl_G$ and $\scl_{G,N}$.

Theorem~\ref{thm:abelian} provides the following example: for $(G,N)=(\ZZ\wr \RR,\bigoplus_{\RR}\ZZ)$, we have $\sup_{x\in [G,N]}\cl_G(x)=\infty$ but $\scl_G \equiv \scl_{G,N}\equiv 0$ on $[G,N]$; see Example~\ref{example:wreath} for more details.
It is a much more difficult problem to construct a group $G$ with $\sup_{x\in [G,G]}\cl_G(x)=\infty$ but $\scl_{G}\equiv 0$ on $[G,G]$ such that the abelianization $G^{\mathrm{ab}}=G/[G,G]$ is finite.
See \cite{Muranov} and \cite{Mimura} for such examples.

\subsection{The class $\CSD$ and permutational wreath products}\label{subsec:perm}
In the final part of this introduction,  we state our result for (possibly) non-abelian $\Gamma$.  The main task is to bound $\cl_G(x)$ \emph{from above} in certain good situations. For this purpose, we define a class of groups $\CSD$ for a non-empty set $D$ with
\[
D\subseteq \{(g,r)\in \NN^2\;|\; g+1\leq r\leq 2g\},
\]
which is related to surface groups; see Definition~\ref{def:CSD}. 
Here, we consider \emph{permutational} (restricted) wreath products: let $\Gamma$ be a group and $\rho\colon \Gamma \curvearrowright X$ be a $\Gamma$-action on a set $X$. Let $H$ be a group. Then the \emph{permutational wreath product} $H\wr_{\rho}\Gamma$ is the semidirect product $\left(\bigoplus_{X}H\right)\rtimes \Gamma$, where $\Gamma$ acts on $\bigoplus_{X}H$ by shifts via $\rho$.

\begin{thm}[result for permutational wreath products]\label{thm:surface_new}
Let $D$ be a non-empty subset of $\{(g,r)\in \NN^2\;|\; g+1\leq r\leq 2g\}$. Assume that a group $\Gamma$ is a member of $\CSD$, the class defined in Definition~$\ref{def:CSD}$. Then, there exist a group quotient $Q$ of $\Gamma$ and  a  quotient map $\sigma\colon \Gamma\twoheadrightarrow Q$ such that
\begin{eqnarray}\label{eq:perm}
(G,N)=(\ZZ\wr_{\rho_Q}\Gamma,\bigoplus_{Q}\ZZ)
\end{eqnarray}
satisfies the following condition, where $\rho_Q\colon G\curvearrowright Q$ is the composition of $\sigma$ and the natural $Q$-action on $Q$ by left multiplication: for every $(g,r)\in D$, there exists $x_{(g,r)}\in [G,N]$ such that
\begin{eqnarray}\label{eq:surface_ineq}
 \left\lceil \frac{r}{2}\right\rceil\leq \cl_{G}(x_{(g,r)})\leq g \quad \textrm{and}\quad \cl_{G,N}(x_{(g,r)})=r.
\end{eqnarray}

In particular, $\cl_{G}$ and $\cl_{G,N}$ do \emph{not} coincide on $[G,N]$ for the pair $(G,N)$ in \eqref{eq:perm}, which  fits  in \emph{split} short exact sequence $(\star)$ above. If  for a fixed real number $C\in [1,2)$,
\[
\sup_{(g,r)\in D}(Cr-g)=\infty
\]
holds, then we have
\[
\sup_{x\in [G,N]}(\cl_{G,N}(x)-C\cdot \cl_G(x))=\infty
\]
for the pair $(G,N)$ above.
\end{thm}
See also Proposition~\ref{prop:surface_key}, which is the key to the construction of $x_{(g,r)}$.

By \eqref{eq:surface_ineq}, if $r\in \{2g-1,2g\}$, then we have
\[
\cl_G(x_{(g,r)})=g.
\]

We present several examples of groups in the class $\CSD$. Our example includes the fundamental groups of mapping tori of certain surface diffeomorphisms. For $g\in \NN$, let $\Mod(\Sigma_g)$ be the mapping class group of $\Sigma_g$.
Here, $\Sigma_g$ denotes the closed connected orientable surface of genus $g$.
We have the \emph{symplectic representation} $s_g\colon \Mod(\Sigma_g)\twoheadrightarrow \Sp(2g,\ZZ)$, which is induced by the action of $\Mod(\Sigma_g)$ on the abelianization of $\pi_1(\Sigma_g)$. For an orientation-preserving diffeomorphism $f\colon \Sigma_g\to \Sigma_g$, let $T_f$ denote the mapping torus of $f$; see Subsection~\ref{subsec:mapping} for details. See also Corollary~\ref{cor:free} and Theorem~\ref{thm:mapping_tori} for further examples. Here for (3) of Theorem \ref{thm:ex_CSD}, we use results by Levitt--Metaftsis \cite{LM} and Amoroso--Zannier \cite{AZ}.

\begin{thm}[examples of groups in $\CSD$]\label{thm:ex_CSD}
\begin{enumerate}[$(1)$]
  \item Every group $\Gamma$ admitting an abelian subgroup $\Lambda$ that is not locally cyclic is a member of $\CC_{\Surf_{\{(1,2)\}}}$.
  \item For every $g\in \NN$, the surface group $\pi_1(\Sigma_g)$ is a member of $\CC_{\Surf_{\{(g,2g)\}}}$. For a non-empty set $J\subseteq \NN$, the free product $\bigast_{g\in J}\pi_1(\Sigma_g)$ is a member of $\CC_{\Surf_{D_J}}$, where $D_J=\{(g,2g)\;|\; g\in J\}$.
  \item There exists an effective absolute constant $n_0\in \NN$ such that the following holds. Assume that $\psi\in \Mod(\Sigma_2)$ satisfies that the order of $s_2(\psi)$ is infinite. Let $f\colon \Sigma_2\to \Sigma_2$ be a diffeomorphism on $\Sigma_2$ whose isotopy class is $\psi$. Then for every $n\in \NN$ with $n\geq n_0$, we have
\[
\pi_1(T_{f^n})\in \CC_{\Surf_{\{(2,3)\}}}\cup \CC_{\Surf_{\{(2,4)\}}}.
\]
  \item Let $g$ be an integer at least $2$. Assume that $\psi\in \Mod(\Sigma_g)$ satisfies that $s_g(\psi)\in \{\pm I_{2g}\}$, where $I_{2g}$ denotes the identity matrix in $\Sp(2g,\ZZ)$. Let $f$ be a diffeomorphism on $\Sigma_g$ whose isotopy class is $\psi$. Then we have
\[
\pi_1(T_{f})\in \CC_{\Surf_{\{(g,2g)\}}}.
\]
\end{enumerate}
\end{thm}

The present paper is organized as follows. In Section~\ref{sec:proof_new}, we prove  Theorems~\ref{thm:wreath_shin} and \ref{thm:split} (hence, Theorem~\ref{thm:wreath_new} as well). In Section~\ref{sec:proof}, we prove   Theorems \ref{thm local cyclicity} and \ref{thm:abelian}. In Section~\ref{sec:surface},  we define the class $\CSD$ and prove Theorem~\ref{thm:surface_new}. In Section~\ref{sec:surface_example}, we discuss examples of groups in $\CSD$, including ones from mapping tori of certain surface diffeomorphisms. Theorem~\ref{thm:ex_CSD} is proved there.  We make concluding remarks in Section~\ref{sec:remark}:  there, we exhibit examples from symplectic geometry  (Subsection~\ref{subsec:symp}) and we collect basic properties of general rank and provide some examples of groups of finite general rank (Subsection~\ref{subsec:gen_rk}).

\section{Proofs of Theorem~\ref{thm:wreath_shin} and Theorem~\ref{thm:split}}\label{sec:proof_new}
\subsection{Proof of Theorem~\ref{thm:wreath_shin}}\label{subsec:wreath}
In this subsection, we study $\cl_{G,N}$ for the pair $(G,N)=(A\wr\Gamma,\bigoplus_{\Gamma}A)$, where $A$ is an additive group and $\Gamma$ is a (possibly non-abelian) group. In this setting, we write $(v,\gamma)$ to indicate the corresponding element of $G$, where $v\in N$ and $\gamma\in \Gamma$. Also, the group $\Gamma$ acts on $N$ as the left-regular representation: for $\lambda\in \Gamma$ and for $u\in N$, $\lambda u\in N$ is a function defined as
\[
\lambda u(\gamma):=u(\lambda^{-1}\gamma)
\]
for $\gamma\in \Gamma$. This action is used to define $G=A\wr \Gamma$. We observe that $\Gamma$ can also act on $N$ as the right-regular representation: for $\lambda\in \Gamma$ and for $u\in N$, $u\lambda \in N$ is a function defined as
\[
u\lambda(\gamma):=u(\gamma\lambda^{-1})
\]
for $\gamma\in \Gamma$. We note that these two actions commute.

\begin{lem}\label{lem:commutator}
Let $A$ be an additive group and $\Gamma$ be a group. Let $G=A\wr \Gamma$, and $N=\bigoplus_{\Gamma}A$. Then for every $v,w\in N$ and every $\gamma,\lambda\in \Gamma$, we have
\[
\big[ (v,\gamma), (w, \lambda)\big] = (\gamma w -[\gamma, \lambda]w +v-\gamma \lambda\gamma^{-1}v, [\gamma,\lambda]).
\]
In particular, if $\gamma$ and $\lambda$ commute, then
\begin{eqnarray} \label{eq mixed commutator_general}
\big[ (v,\gamma), (w, \lambda)\big] = (\gamma w - w +v-\lambda v, e_\Gamma).
\end{eqnarray}
Also, we have
\begin{eqnarray} \label{eq mixed commutator}
\big[ (v,\gamma), (w, e_{\Gamma})\big] = (\gamma w - w, e_\Gamma).
\end{eqnarray}
\end{lem}

\begin{proof}
Let $(v,\gamma), (v', \gamma'), (w, \lambda) \in G$. Then we have
\[ (v, \gamma) (v', \gamma') = (v + \gamma v', \gamma \gamma' ), \; (v, \gamma)^{-1} = (- \gamma^{-1} v, \gamma^{-1}).\]
Using these, we have
\begin{eqnarray*}
\big[ (v, \gamma), (w, \lambda) \big] & = & (v, \gamma) (w, \lambda) (- \gamma^{-1} v, \gamma^{-1}) (- \lambda^{-1} w, \lambda^{-1}) \\
& = & (v + \gamma w, \gamma \lambda) (- \gamma^{-1} v, \gamma^{-1})(- \lambda^{-1} w, \lambda^{-1}) \\
& = & (v + \gamma w - \gamma \lambda \gamma^{-1} v , \gamma \lambda \gamma^{-1})  (- \lambda^{-1} w, \lambda^{-1})\\
& = & (v + \gamma w - \gamma \lambda \gamma^{-1} v - [\gamma,\lambda]w , [\gamma,\lambda]),
\end{eqnarray*}
as desired. This immediately implies \eqref{eq mixed commutator_general} and \eqref{eq mixed commutator}.
\end{proof}

The following lemma is a key to bounding mixed commutator lengths from below.

\begin{lem}\label{lem:claim}
Let $A$ be an  additive group and $\Gamma$ be a group. Let $G=A\wr \Gamma$, and $N=\bigoplus_{\Gamma}A$. Assume that $u\in N$ is written as
\[
u=\sum_{i=1}^k (\gamma_i w_i-w_i)
\]
for some $k\in \NN$, some $\gamma_1,\ldots ,\gamma_k\in \Gamma$ and some $w_1,\ldots ,w_k\in N$. Let $\Lambda=\langle \gamma_1,\ldots ,\gamma_k\rangle$. Then we have
\[
\sum_{\lambda\in\Lambda}u(\lambda)=0.
\]
\end{lem}

\begin{proof}
It suffices to show that
\[ \sum_{\lambda \in \Lambda} (\gamma_i w_i - w_i)(\lambda) = 0 \]
for every $i \in \{ 1, \cdots, k\}$. Since $\gamma_i \in \langle \gamma_1, \cdots, \gamma_k \rangle$, we conclude that the correspondence $\Gamma \to \Gamma$, $\gamma \mapsto \gamma_i^{-1} \gamma$ sends $\Lambda$ to $\Lambda$ bijectively. This clearly verifies the assertion above.
\end{proof}

In what follows, we mainly discuss the case where $A=\ZZ$. In this case, we recall from the introduction that $\delta_\lambda \colon \Gamma \to \ZZ$ for $\lambda\in \Gamma$ denotes   the Kronecker delta function at $\lambda$. Then, the left and right actions of $G$ on $N$ are expressed as $\gamma \delta_{\lambda} =\delta_{\gamma\lambda}$ and $\delta_{\lambda}\gamma =\delta_{\lambda\gamma}$ for $\gamma,\lambda\in \Gamma$. The following proposition provides the estimate in  Theorem~\ref{thm:wreath_shin}~(1) from below.

\begin{prop}\label{prop:frombelow}
Let  $\Gamma$ be a group. Set $G=\ZZ\wr \Gamma$, and $N=\bigoplus_{\Gamma}\ZZ$. Let $k\in \NN$.  Let $\Lambda$ be a subgroup of $\Gamma$. Assume that $x\in N$ fulfills the following two conditions:
\begin{itemize}
  \item $e_{\Gamma}\in \supp(x)$.
  \item For every zero-sum set $S$ for x satisfying $S\subseteq \supp(x)$ and $e_{\Gamma}\in S$, we have $\langle S \rangle \geqslant \Lambda$.
\end{itemize}
Then, we have
\[
\cl_{G,N}(x)\geq \intrk^{\Gamma}(\Lambda).
\]
Here, we set $\cl_{G,N}(x)=\infty$ if $x\not \in [G,N]$.
\end{prop}

\begin{proof}
If $x\not \in [G,N]$, then the inequality trivially holds. In what follows, we assume that $x\in [G,N]$. Set $l=\cl_{G,N}(x)$. Then, by \eqref{eq mixed commutator}, there exist  $\gamma_1, \cdots, \gamma_l \in \Gamma$ and $w_1, \cdots, w_l \in N$ such that
\[
x = \sum_{i=1}^l (\gamma_i w_i - w_i).
\]
Set $\Theta = \langle \gamma_1, \cdots, \gamma_l \rangle$. Then Lemma~\ref{lem:claim} implies that
\begin{eqnarray} \label{eq Theta}
\sum_{\theta \in \Theta} x(\theta) = 0.
\end{eqnarray}
Set $S=\supp(x)\cap \Theta$. Then by \eqref{eq Theta}, $S$ is a zero-sum set for $x$ that satisfies $S\subseteq \supp(x)$ and $e_{\Gamma}\in S$. Hence by assumption, $\langle S\rangle \geqslant \Lambda$ holds. This implies that
\[
\Theta \geqslant \Lambda.
\]
Therefore, we have
\[
l\geq \rk(\Theta)\geq \intrk^{\Gamma}(\Lambda). \qedhere
\]
\end{proof}

The next proposition, in turn, provides the estimate in  Theorem~\ref{thm:wreath_shin}~(1) from above.

\begin{prop}\label{prop:fromabove}
Let  $\Gamma$ be a group. Let $G=\ZZ\wr \Gamma$, and $N=\bigoplus_{\Gamma}\ZZ$. Let $\Theta$ be a non-trivial finitely generated subgroup of $\Gamma$.  Assume that an element $x\in N$ satisfies the following two conditions:
\begin{itemize}
  \item $\supp(x)\subseteq \Theta$.
  \item $\Theta$ is a zero-sum set for $x$.
\end{itemize}
Then, we have  $x\in [G,N]$ and
\[
\cl_{G,N}(x)\leq \rk(\Theta).
\]
\end{prop}

\begin{proof}
Let $r=\rk(\Theta)\in \NN$. Fix elements $\theta_1,\ldots ,\theta_r\in \Theta$ that satisfy $\langle \theta_1,\ldots ,\theta_r\rangle =\Theta$. Define $M$ as
\[
M=\left\{\sum_{i=1}^r (\theta_iw_i-w_i) \; \middle|\; w_1,w_2,\ldots ,w_r\in N \right\}.
\]
By \eqref{eq mixed commutator}, every element $z$ in $M$ satisfies that
\[
\cl_{G,N}(z)\leq r.
\]
In what follows, we will show that $x\in M$.

Observe that $M$ is a $\ZZ$-module equipped with the right $\Gamma$-action $z\mapsto z \gamma$ for $z\in M$ and $\gamma \in \Gamma$. To see that $M$ is closed under the right action, write $z\in M$ as $\sum_{i=1}^r (\theta_iw_i-w_i)$, where $w_1,\ldots,w_r\in N$. Then, since the left and right actions of $\Gamma$ on $N$ commute, we conclude that  for every $\gamma\in \Gamma$,
\[
z\gamma=\sum_{i=1}^r (\theta_i(w_i\gamma)-(w_i\gamma))\in M.
\]
We claim that for every $\theta \in \Theta$,
\begin{eqnarray}\label{eq:induction}
\delta_{\theta}-\delta_{e_{\Gamma}}\in M
\end{eqnarray}
holds true. We verify this claim by induction on the word length of $\theta$ with respect to $\{\theta_1,\ldots ,\theta_r\}$. If $\theta=e_{\Gamma}$ or $\theta\in \{\theta_1^{\pm},\ldots ,\theta_r^{\pm}\}$, then \eqref{eq:induction} holds. Indeed, to see the case where $\theta\in \{\theta_1^{-1},\ldots ,\theta_r^{-1}\}$,  we have for every $1\leq i\leq r$,
\[
\delta_{\theta_i^{-1}}-\delta_{e_{\Gamma}}=(\delta_{e_{\Gamma}}-\delta_{\theta_i})\theta_i^{-1}=(\delta_{e_{\Gamma}}-\theta_i\delta_{e_{\Gamma}})\theta_i^{-1}\in M.
\]
In our induction step, take an arbitrary element $\theta\in \Theta$ whose word length with respect to $\{\theta_1,\ldots ,\theta_r\}$ is at least $2$. Write $\theta$ as
\[
\theta= \theta'\lambda,
\]
where $\lambda\in \{\theta_1^{\pm},\ldots ,\theta_r^{\pm}\}$ and the word length of $\theta'$ with respect to $\{\theta_1,\ldots ,\theta_r\}$ is smaller than that of $\theta$. Our induction hypothesis implies that
\[
\delta_{\theta'}-\delta_{e_{\Gamma}}\in M.
\]
We then have
\begin{eqnarray*}
\delta_{\theta}-\delta_{e_{\Gamma}} &=& \delta_{\theta'\lambda}-\delta_{e_{\Gamma}}\\
&=& (\delta_{\theta'\lambda}-\delta_{\lambda})+(\delta_{\lambda}-\delta_{e_{\Gamma}})\\
&=& (\delta_{\theta'}-\delta_{e_{\Gamma}})\lambda+(\delta_{\lambda}-\delta_{e_{\Gamma}})\in M;
\end{eqnarray*}
this ends our proof of the claim.

Finally,  we have
\[
x=\sum_{\theta\in \Theta}x(\theta)(\delta_{\theta}-\delta_{e_{\Gamma}})
\]
by assumption.  Therefore, by the claim above (see \eqref{eq:induction}) we conclude that
\[
x\in M.
\]
This completes the proof.
\end{proof}

To prove Theorem~\ref{thm:wreath_shin}~(2), we employ the following result on general and intermediate ranks. We include the proof for the convenience of the reader.

\begin{lem}\label{lem:gen_rk}
Let $\Gamma$ be a group. Assume that $r\in \NN$ satisfies $r\leq \genrk(\Gamma)$. Then, there exists a finitely generated subgroup $\Lambda_r$ of $\Gamma$ such that
\[
r=\intrk^{\Gamma}(\Lambda_r).
\]
\end{lem}

\begin{proof}
The conclusion holds by definition if $r=\genrk(\Gamma)$. In what follows, we treat the remaining case: in this case, there exists a finitely generated subgroup $\Lambda$ of $\Gamma$ such that $\intrk^{\Gamma}(\Lambda)>r$. Set $s=\intrk^{\Gamma}(\Lambda)$. There exists a subgroup $\Theta$  of $\Gamma$ with $\Lambda\leqslant \Theta\leqslant \Gamma$ and $\rk(\Theta)=s$. Fix a set of generators $\{\theta_1,\ldots,\theta_s\}$ of size $s$ of $\Theta$. Set
\[
\Lambda_r=\langle \theta_1,\ldots,\theta_r\rangle.
\]
We claim that
\begin{eqnarray}\label{eq:r}
\intrk^{\Gamma}(\Lambda_r)=r.
\end{eqnarray}
To prove this claim, suppose that it is not the case. Then, there exists a subgroup $\Theta_r$ of $\Gamma$ with $\Lambda_r\leqslant \Theta_r\leqslant \Gamma$ such that $\rk(\Theta_r)<r$. Set
\[
\overline{\Theta}=\langle \Theta_r \cup \{\theta_{r+1}.\ldots ,\theta_s\}\rangle.
\]
Then, we have $\Lambda\leqslant \Theta\leqslant \overline{\Theta}\leqslant \Gamma$ but
\[
\rk(\overline{\Theta})\leq \rk(\Theta_r)+(s-r)<s;
\]
this contradicts $\intrk^{\Gamma}(\Lambda)=s$. Therefore, we obtain \eqref{eq:r}.
\end{proof}

Now we are ready to prove Theorem~\ref{thm:wreath_shin}.

\begin{proof}[Proofs of Theorems~$\ref{thm:wreath_shin}$ and $\ref{thm:wreath_new}$]
First, we show Theorem~\ref{thm:wreath_shin}~(1). Set $r=\intrk^{\Gamma}(\Lambda)$; note that $r\leq k<\infty$. Take $\Theta$ with $\Lambda\leqslant \Theta\leqslant \Gamma$ that satisfies $\rk(\Theta)=r$. Then, by conditions (i), (ii), and (iii), we can apply Propositions~\ref{prop:frombelow} and \ref{prop:fromabove} to $x$. Indeed, since $\supp(x)$ is a zero-sum set with $\supp(x)\subseteq \supp(x)$ and $e_{\Gamma}\in \supp(x)$ by (i) and (ii), condition (iii) implies that
\[
\supp(x)\subseteq \langle\supp(x)\rangle =\Lambda.
\]
Then, these propositions provide the estimates
\[
\cl_{G,N}(x)\geq r\quad \textrm{and}\quad \cl_{G,N}(x)\leq r,
\]
respectively. Therefore, we obtain the conclusion of Theorem~\ref{thm:wreath_shin}~(1).

Next, we will prove Theorem~\ref{thm:wreath_shin}~(2). Let $r\in \NN$ satisfy $r\leq \genrk(\Gamma)$. By Lemma~\ref{lem:gen_rk}, there exists a finitely generated subgroup $\Lambda_r$ of $\Gamma$ such that $\intrk^{\Gamma}(\Lambda_r)=r$. Take a generating set $\{\lambda_1,\ldots ,\lambda_r\}$ of $\Lambda_r$ of size $r$. Set
\[
x_r=\sum_{i=1}^r \delta_{\lambda_i}-r\delta_{e_{\Gamma}}\in N.
\]
Then, this $x_r$ fulfills conditions (i)--(iii) of Theorem~\ref{thm:wreath_shin}~(1) with $\Lambda$ being $\Lambda_r=\langle \supp(x_r)\rangle$. In particular, Theorem~\ref{thm:wreath_shin}~(1) implies that
\[
\cl_{G,N}(x_r)=r.
\]
The rest is to show that
\begin{eqnarray}\label{eq:reversed_incl}
\{\cl_{G,N}(x)\;|\; x\in [G,N]\}\subseteq \{r\in \ZZ_{\geq 0}\;|\; r\leq \genrk(\Gamma)\}.
\end{eqnarray}
Take an arbitrary $x\in [G,N]$. Then $\supp(x)$ is finite and $\supp(x)$ is a zero-sum set for $x$. Define $\Lambda_x:=\langle \supp(x)\rangle$, which is a finitely generated subgroup of $\Gamma$. Then, Proposition~\ref{prop:fromabove} implies that $\cl_{G,N}(x)\leq \intrk^{\Gamma}(\Lambda_x)$. Hence, we obtain \eqref{eq:reversed_incl}. This completes the proof of Theorem~\ref{thm:wreath_shin}~(2).

Finally, Theorem~\ref{thm:wreath_new} immediately follows from Theorem~\ref{thm:wreath_shin}~(1).
\end{proof}

Theorem~\ref{thm:wreath_shin}~(2) (former assertion) can be generalized to the following setting. Here, recall from Subsection~\ref{subsec:perm} the definition of permutational wreath products.

\begin{prop}\label{prop:wreath_perm}
Let $\Gamma$ be a group. Let $Q$ be a group quotient of $\Gamma$ and $\sigma\colon \Gamma \twoheadrightarrow Q$ be a group quotient map. Let $\rho_Q\colon \Gamma\curvearrowright Q$ be the $\Gamma$-action defined by the composition of $\sigma$ and the natural $Q$-action $Q\curvearrowright Q$ by left multiplication. Set
\[
(G,N)=(\ZZ\wr_{\rho_Q}\Gamma,\bigoplus_{Q}\ZZ).
\]
Then we have
\[
\{\cl_{G,N}(x)\;|\;x\in [G,N]\}=\{r\in \ZZ_{\geq 0}\;|\; r\leq \genrk(Q)\}.
\]
\end{prop}

\begin{proof}
For $v\in N$ and $\gamma\in \Gamma$, write $(v,\gamma)$ to indicate the corresponding element of $G$. Note that $N$ adimits a narutal left $Q$-action. Then, similar computation to one in Lemma~\ref{lem:commutator} shows that for every $w,v\in N$ and every $\gamma\in \Gamma$,
\[
\big[ (v,\gamma), (w, e_{\Gamma})\big] = (\sigma(\gamma) w - w , e_{\Gamma}).
\]
Observe that $N$ admits the right $Q$-action $vq(p):=v(pq^{-1})$ for $v\in N$ and $p,q\in Q$, which commutes with the left $Q$-action. Therefore, the proofs of Propositions~\ref{prop:frombelow} and \ref{prop:fromabove} can be generalized to our setting.
\end{proof}

\subsection{Proof of Theorem~\ref{thm:split}}\label{subsec:split}
Here we prove Theorem~\ref{thm:split}. The proof goes along a similar way to the one in the proof of \cite[Theorem~7.1]{KKMM1}; we include the proof for the convenience of the reader. The following lemma is the key (see also \cite[Lemma~7.4]{KKMM1}).

\begin{lem}\label{lem:split}
Let $q\colon G\twoheadrightarrow \Gamma$ be a group quotient map and $N=\Ker(q)$. Assume that $N$ is \emph{abelian}.
Let $k\in \NN$. Let $f_1, \cdots, f_k, g_1, \cdots, g_k, a_1, \cdots, a_k, b_1, \cdots, b_k \in G$ with
\[
q(f_i) = q(a_i)
\quad \textrm{and}\quad
q(g_i) = q(b_i)
\]
for every $ 1\leq i \leq k$.
Then, $[f_1, g_1] \cdots [f_k, g_k] ([a_1, b_1] \cdots [a_k, b_k])^{-1}$ is contained in $[G,N]$. Moreover, the following inequality holds:
\[
\cl_{G,N}\Big( [f_1,g_1] \cdots [f_k,g_k] \big( [a_1, b_1] \cdots [a_k, b_k] \big)^{-1} \Big) \leq 2k.
\]
\end{lem}
\begin{proof}
We prove this lemma by induction on $k$. The case $k = 0$ is obvious. Suppose that $k = 1$, and set $f = f_1$, $g = g_1$, $a = a_1$, and $b = b_1$. Then, there exist $v_1, v_2 \in N$ such that $f = a v_1$ and $g = a v_2$. Then we have
\begin{eqnarray*}
[f,g][a, b]^{-1} & = & [a v_1, b v_2] [a, b]^{-1} \\
& = & a v_1 b v_2 v_1^{-1} a^{-1} v_2^{-1} b^{-1} b a b^{-1} a^{-1} \\
& = & a v_1 b v_2 v_1^{-1}  a^{-1} v_2^{-1} a b^{-1} a^{-1} \\
& = & a b (b^{-1} v_1 b v_1^{-1} v_1 v_2 v_1^{-1} v_2^{-1} v_2 a^{-1} v_2^{-1} a) b^{-1} a^{-1} \\
& = & (a b) [b^{-1}, v_1] [v_2, a^{-1}] (a b)^{-1}.
\end{eqnarray*}
Here, recall that $N$ is assumed to be abelian  and hence that $v_1v_2v_1^{-1}v_2^{-1}=e_G$.
Note that $\cl_{G,N}$ is $G$-invariant, meaning that
\[
\cl_{G,N}(x)=\cl_{G,N}(zxz^{-1})
\]
for every $z\in G$ and every $x\in N$: this is because
\[
z[g,v]z^{-1}=[zgz^{-1},zvz^{-1}]
\]
for every $z,g\in G$ and every $v\in N$.
Therefore, we  have $\cl_{G,N}([f,g] [a, b]^{-1}) \le 2$. Now, we proceed to the induction step. Suppose that $k \geq 2$, and set
\[
\xi = [a_1, b_1] \cdots [a_{k-1}, b_{k-1}].
\]
Then we have
\[
[f_1, g_1] \cdots [f_k, g_k] ([a_1, b_1] \cdots [a_k, b_k])^{-1} = ([f_1, g_1] \cdots [f_{k-1}, g_{k-1}] \xi^{-1}) \cdot (\xi [f_k, g_k] [a_k, b_k]^{-1} \xi^{-1}).
\]
By $G$-invariance of $\cl_{G,N}$ and the inductive hypothesis, we conclude that
\[
\cl_{G,N} \Big( [f_1, g_1] \cdots [f_k, g_k] ([a_1, b_1] \cdots [a_k, b_k])^{-1}\Big) \le 2(k-1)+2=2k.
\]
This completes the proof.
\end{proof}

\begin{proof}[Proof of Theorem~$\ref{thm:split}$]
In split exact sequence $(\star)$,
take a homomorphism $s\colon \Gamma\to G$ such that $q\circ s={\rm id}_{\Gamma}$.
Take an arbitrary element $x$ in $[G,N]$. Then, Lemma~\ref{lem:split} implies that
\[
\cl_{G}(x \cdot ((s\circ q)(x))^{-1})\leq 2\cl_{G,N}(x).
\]
Here, note that $s\circ q\colon G\to G$ is a group homomorphism. Observe that $(s\circ q)(x)=e_{G}$ since $[G,N]\leqslant N=\Ker(q)$. Therefore, we conclude that
\[
\cl_{G}(x)\leq 2\cl_{G,N}(x),
\]
as desired.
\end{proof}

\section{The case of abelian $\Gamma$}\label{sec:proof}
 The goal  of this section is to prove Theorems~\ref{thm local cyclicity} and \ref{thm:abelian}.
In this section, we focus on the case where $\Gamma$ is abelian.  First, we recall the definition of \emph{special rank} of groups in the sense of Malcev \cite{Mal}.

\begin{definition}[\cite{Mal}]\label{def:local_rank}
For a group $\Gamma$, the \emph{special rank} of $\Gamma$ is defined by
\[
\sperk(\Gamma)=\sup \{ \rk(\Lambda) \; | \; \textrm{$\Lambda$ is a finitely generated subgroup  of $\Gamma$} \}\  \in \ZZ_{\geq 0}\cup \{\infty\}.
\]
\end{definition}

For a group pair $(\Gamma,\Lambda)$ with $\Gamma\geqslant \Lambda$, computation of $\intrk^{\Gamma}(\Lambda)$ is easy if $\Gamma$ is abelian.

\begin{lem}\label{lem:abel}
Let $(\Gamma,\Lambda)$ be a pair of a group $\Gamma$ and its subgroup $\Lambda$. Assume that $\Gamma$ is abelian. Then we have that
\[
\intrk^{\Gamma}(\Lambda)= \rk(\Lambda).
\]
In particular, we have that
\[
\genrk(\Gamma)=\sperk(\Gamma).
\]
\end{lem}

\begin{proof}
The inequality $\intrk^{\Gamma}(\Lambda)\leq  \rk(\Lambda)$ holds in general. Take a group $\Theta$ with $\Lambda\leqslant \Theta\leqslant \Gamma$. If $\rk(\Theta)=\infty$, then $\rk(\Lambda)\leq \rk(\Theta)=\infty$. If $\rk(\Theta)<\infty$, then the classification of finitely generated abelian groups implies
\[
\rk(\Theta)\geq \rk(\Lambda).
\]
Hence, we conclude that $\rk(\Theta)\geq \rk(\Lambda)$. Therefore, we have the reversed inequality
\[
\intrk^{\Gamma}(\Lambda)\geq  \rk(\Lambda)
\]
as well. This ends our proof.
\end{proof}

See Subsection~\ref{subsec:gen_rk} for more details on general ranks and special ranks.

In what follows, we will prove Theorem~\ref{thm local cyclicity}, which states that $\cl_G$ and $\cl_{G,N}$ coincide on $[G,N]$ when $\Gamma$ is locally cyclic.
Our proof of Theorem~\ref{thm local cyclicity} employs the following lemma.

\begin{lem}\label{lem:commcomp}
Let $G$ be a group and let $g,h \in G$. Then the following hold:
\begin{enumerate}[$(1)$]
\item If $x \in G$ commutes with $g$, then $[g, hx] = [g,h]$. In particular, we have $[g,h] = [g, hg^k]$ for every $k \in \ZZ$.

\item If $y \in G$ commutes with $h$, then $[gy, h] = [g,h]$. In particular, we have $[g,h] = [gh^k,h]$ for every $k \in \ZZ$.
\end{enumerate}
\end{lem}
\begin{proof}
Computation shows that
\[ [g,hx] = ghx g^{-1} x^{-1} h^{-1} = ghg^{-1} h^{-1} = [g,h],\]
\[ [gy, h] = gyh y^{-1} g^{-1} h^{-1} = ghg^{-1} h^{-1} = [g,h].\qedhere\]
\end{proof}

\begin{remark} \label{remark left right}
Using Lemma~\ref{lem:commcomp}, we have  for every $g\in G$ and every $x\in N$,
\[ [g,x] = [g, xg^{-1}] = [gxg^{-1}, x g^{-1}], \; [x,g] = [xg^{-1}, g] = [xg^{-1}, gxg^{-1}].\]
These mean that $c \in G$ is a $(G,N)$-commutator if and only if there exist $g \in G$ and $x \in N$ such that $c = [x,g]$.
\end{remark}

\begin{proof}[Proof of Theorem~$\ref{thm local cyclicity}$]
Let $g,h \in G$.
We show  that $[g,h]$ is a $(G,N)$-commutator.
Without loss of generality, we may assume that $(\bar{g},\bar{h})\neq (e_\Gamma,e_\Gamma)$.
Let  $q$ denote the quotient map $G \to \Gamma$. We write $\bar{x}$ instead of $q  (x)$ for $x \in G$. Since $\Gamma$ is locally cyclic  and $(\bar{g},\bar{h})\neq (e_\Gamma,e_\Gamma)$,  there exists an isomorphism $f$ from  $\langle \bar{g}, \bar{h} \rangle$ to $\ZZ$ or $\ZZ/k\ZZ$ for some $k > 0$.

We first show the case $\langle \bar{g}, \bar{h} \rangle$ is isomorphic to $\ZZ$. We now show the following claim:

\bigskip \noindent
{\bf Claim~1.} Assume that $g_0$ and $h_0$ are elements in $G$  such that $g_0, h_0 \in \langle g, h \rangle$ and $[g_0, h_0] = [g,h]$. If $\min \{ |f (\bar{g_0})|, |f(\bar{h}_0)| \} > 0$, then there exist $g_1 , h_1 \in G$ such that $g_1, h_1 \in \langle g, h \rangle$, $[g_1, h_1] = [g,h]$, and
\[ \min \{ |f(\bar{g}_0)|, |f(\bar{h}_0)|\} > \min \{ |f(\bar{g}_1)|, |f(\bar{h}_1)|\}. \]

Now we start the proof of Claim~1. Set $m = |f(\bar{g}_0)|$ and $n = |f(\bar{h}_0)|$. Suppose $m \ge n$. Then there exist $k,r \in \ZZ$ such that $m = nk + r$ and $|r| < |n|$.

Set $g_1 = g_0 h_0^{-k}$ and $h_ 1 = h_0$.  By Lemma \ref{lem:commcomp}, we have
\[ [g,h] = [g_0, h_0] = [g_0 h_0^{-k}, h_0] = [g_1, h_1].\]
Then $f(\bar{g}_1) = f(  q  (g_0 h_0^{-k})) = m - kn = r$. Thus we have
\[ \min \{ |f(\bar{g}_1)|, |f(\bar{h}_1)|\} = \min \{ |r|, |n|\} = |r| < |n| = \min \{ |m|, |n|\}.\]
The case $m \le n$ is  proved in a similar manner. This completes the proof of our claim.

Starting with the case $g = g_0$ and $h = h_0$, applying  Claim~1 iteratively, we obtain $g', h' \in G$ such that $[g,h] = [g', h']$ and one of $g'$ and $h'$ belongs to $N$. Hence  Remark~\ref{remark left right} completes the proof of the case $\langle \bar{g}, \bar{h} \rangle$.

Next we consider the case $f$ is an isomorphism from $\langle \bar{g}, \bar{h} \rangle$ to $\ZZ/k\ZZ$.
Let $\tilde{f}$ denote the composition $\langle \bar{g}, \bar{h} \rangle \to \ZZ / k \ZZ \xrightarrow{\cong} \{ 0,1, \cdots, k-1\}$. Here $\ZZ / k \ZZ \to \{ 0,1, \cdots, k-1\}$ is the inverse of the natural projection $\{ 0,1, \cdots, k-1\} \to \ZZ/k\ZZ$.
In a similar manner to Claim~1, we can show the following:

\bigskip \noindent
{\bf Claim~2.} Assume that $g_0$ and $h_0$ are elements in $G$ such that $g_0, h_0 \in \langle g, h \rangle$ and $[g_0, h_0] = [g,h]$. If $\min \{ \tilde{f}(\bar{g}_0), \tilde{f}(\bar{h}_0)\} > 0$, then there exist $g_1, h_1 \in G$ such that $g_1, h_1 \in \langle g, h \rangle$, $[g_1, h_1] = [g,h]$ and
\[ \min \{ \tilde{f}(\bar{g}_0), \tilde{f}(\bar{h}_0)\} > \min \{ \tilde{f}(\bar{g}_1), \tilde{f}(\bar{h}_1)\}.\]

Using Claim 2 iteratively, we can obtain elements $g',h' \in G$ such that $[g,h] = [g', h']$ and one of $g'$ and $h'$ belongs to $N$. Remark~\ref{remark left right} completes the proof.
\end{proof}

 Now we proceed to the proof of Theorem~\ref{thm:abelian}.

\begin{proof}[Proof of Theorem~$\ref{thm:abelian}$]
First we claim that
\[
\cl_{G}(x)=\left\lceil \frac{\cl_{G,N}(x)}{2}\right\rceil
\]
for every $x\in [G,N]$. Indeed, this equality follows from \eqref{eq mixed commutator_general} and \eqref{eq mixed commutator}. Therefore, Theorem~\ref{thm:wreath_shin}~(2) and Lemma~\ref{lem:abel} complete our proof.
\end{proof}

\begin{remark}
For a fixed prime number $p$, let $A = \bigoplus_{n \in \NN} ( \ZZ / p^n \ZZ )$. Then, in the setting of Theorem~\ref{thm:abelian}, we may replace $(G,N)=(\ZZ\wr \Gamma, \bigoplus_{\Gamma}\ZZ)$ with $(G,N) = (A \wr \Gamma, \bigoplus_{\Gamma}A)$.
This new pair $(G,N)$ satisfies the same conclusion as in Theorem~\ref{thm:abelian}. This provides an example with $N$ being \emph{locally finite}, meaning that every finitely generated subgroup of $N$ is finite.
\end{remark}

\begin{example}\label{example:wreath}
Set  $(G,N) = (\ZZ \wr \RR, \bigoplus_{\RR} \ZZ)$, or $(G,N) = (\ZZ \wr \overline{\QQ}^{\rm alg},\bigoplus_{\overline{\QQ}^{\rm alg}} \ZZ)$.
We note that $G$ is countable in the latter example.
Here $\overline{\QQ}^{\rm alg}$ denotes the algebraic closure of $\QQ$.
Then we claim that
\begin{eqnarray}\label{eq:cl}
\sup_{x\in [G,N]}\cl_{G}(x)=\infty \quad \textrm{and}\quad \sup_{x\in [G,N]}\cl_{G,N}(x)=\infty
\end{eqnarray}
and that
\begin{eqnarray}\label{eq:scl}
\scl_{G}\equiv 0 \ \textrm{on}\ [G,N] \quad \textrm{and}\quad \scl_{G,N}\equiv 0 \ \textrm{on}\ [G,N].
\end{eqnarray}
Indeed, Theorem~\ref{thm:abelian} implies \eqref{eq:cl} because $\sperk(\RR)=\sperk(\overline{\QQ}^{\rm alg})=\infty$. Lemma~\ref{lem:commutator}  implies that
\[
\big[ (nv,\gamma), (nw, e_{\Gamma})\big] = \big[ (v,\gamma), (w, e_{\Gamma})\big]^n\]
for every $v,w\in N$, every $\gamma\in \Gamma$ and every $n\in \NN$. From this together with commutativity of $N$, we can deduce that for every $x\in [G,N]$,
\[
\sup_{n\in \NN}\cl_{G}(x^n)\leq \sup_{n\in \NN}\cl_{G,N}(x^n)=\cl_{G,N}(x)<\infty.
\]
Therefore, we obtain \eqref{eq:scl}.

One remark is that we can deduce \eqref{eq:scl} from the Bavard duality theorem for $\scl_{G,N}$ (\cite[Theorem~1.2]{KKMM1}). Indeed, since $N$ is abelian, the Bavard duality theorem for $\scl_{G,N}$ implies that $\scl_{G,N}\equiv 0$ on $[G,N]$. Hence, for every $x\in [G,N]$, we have  $\scl_G(x)\leq \scl_{G,N}(x)=0$.
\end{example}

Here, we provide one application of Theorem~\ref{thm local cyclicity}, which is an improvement of our previous work in \cite{KKMM1}.

\begin{example}\label{ex:braid}
For $n \ge 2$, recall that the \emph{Artin braid group $B_n$ with $n$ strands} is defined to be the group generated by $n-1$ generators $\sigma_1, \cdots, \sigma_{n-1}$ with the following relations:
\[ \sigma_i \sigma_j = \sigma_j \sigma_i\]
for all $i, j \in \{1,2,\ldots ,n-1\}$ satisfying $|i - j| \ge 2$ and
\[ \sigma_i \sigma_{i+1} \sigma_i = \sigma_{i+1} \sigma_i \sigma_{i+1}\]
for every $1\leq i\leq n-2$. For the foundation of braid groups, see \cite{KTbook} for example. Set
\[
(G,N)=(B_n,[B_n,B_n]).
\]
$\Gamma=G/N$ is isomorphic to $\ZZ$, and hence $(\star)$ splits for this triple $(G,N,\Gamma)$. In particular, \eqref{eq:3bai} provides that $\cl_{G,N}(x)\leq 3\cl_{G}(x)$ for every $x\in [G,N]$: this estimate was obtained in \cite[Example 7.14]{KKMM1}. In fact, Theorem~\ref{thm local cyclicity} implies that the genuine equality
\[
\cl_{G,N}(x)=\cl_{G}(x)
\]
holds for every $x\in [G,N]$.
\end{example}

\section{The class $\CSD$}\label{sec:surface}
In this section, we define  a class $\CSD$ of groups and prove Theorem~\ref{thm:surface_new}. Recall that $\NN=\{1,2,3,\ldots\}$ in this paper.

\subsection{Definition of $\CSD$}\label{subsec:classes}
We first introduce the following notion.

\begin{definition}[$\pi_1(\Sigma_g)$-triple]\label{def:triple}
Let $g\in \NN$. Let $\Gamma$ be a group. A \emph{$\pi_1(\Sigma_g)$-triple for $\Gamma$} is defined to be a triple $(\phi,Q,\sigma)$ of a group homomorphism $\phi\colon \pi_1(\Sigma_g)\to \Gamma$, a group quotient $Q$ of $\Gamma$ and a group quotient map $\sigma\colon \Gamma\to Q$ such that $P:=(\sigma\circ \phi)(\pi_1(\Sigma_g))$ is abelian.
\end{definition}

\begin{remark}\label{rem:notinj}
In Definition~\ref{def:triple}, the homomorphism $\phi$ is \emph{not} required to be injective.
\end{remark}

We note that in Definition~\ref{def:triple},
\[
\intrk^Q(P)\leq \rk(P)\leq \rk(\pi_1(\Sigma_g))=2g.
\]

\begin{definition}[class $\CSD$]\label{def:CSD}
Let $D$ be a non-empty subset of
\[
\{(g,r)\in \NN^2\;|\; g+1\leq r\leq 2g\}.
\]
Then, the class $\CSD$ is defined as the class of all groups $\Gamma$ that satisfy the following condition: there exist a group quotient $Q$ of $\Gamma$ and a quotient map $\sigma\colon \Gamma\twoheadrightarrow Q$ such that for every $(g,r)\in D$, there exists $\phi_{(g,r)}\colon \pi_1(\Sigma_g) \to \Gamma$ such that $(\phi_{(g,r)},Q,\sigma)$ is a $\pi_1(\Sigma_g)$-triple for $\Gamma$ satisfying
\begin{eqnarray}\label{eq:ranksurf}
\intrk^{Q}(P_{(g,r)})=r,
\end{eqnarray}
where $P_{(g,r)}=(\sigma\circ \phi_{(g,r)})(\pi_1(\Sigma_g))$.
\end{definition}

By definition, we have
\begin{eqnarray}\label{eq:bigcap}
\CSD\subseteq \bigcap_{(g,r)\in D}\CC_{\Surf_{\{(g,r)\}}}.
\end{eqnarray}
The following proposition asserts that a weaker form of the reversed inclusion also holds.

\begin{prop}\label{prop:reversed}
Let $D\subseteq \{(g,r)\in \NN^2\;|\; g+1\leq r\leq 2g\}$ be non-empty.
Let $\Gamma \in \bigcap_{(g,r)\in D}\CC_{\Surf_{\{(g,r)\}}}$.
Then there exists $D_{\Gamma}\subseteq \{(g,t)\in \NN^2\;|\; g+1\leq t\leq 2g\}$ such that the following hold true:
\begin{itemize}
  \item $\Gamma \in \CC_{\Surf_{D_{\Gamma}}}$,
  \item for every $(g,r) \in D$, there exists $r'$ with $r \leq r' \leq 2g$ such that $(g,r') \in D_{\Gamma}$.
\end{itemize}
\end{prop}

The following lemma is immediate by definition; it is employed in the proof of Proposition~\ref{prop:reversed}.
\begin{lem}\label{lem:quotient}
Let $(\Gamma,\Lambda)$ be a pair of a group $\Gamma$ and its subgroup $\Lambda$. Let $\tau\colon \Gamma\twoheadrightarrow \tau(\Gamma)$ be a group quotient map. Then we have
\[
\intrk^{\tau(\Gamma)}(\tau(\Lambda))\leq \intrk^{\Gamma}(\Lambda).
\]
\end{lem}

\begin{proof}[Proof of Proposition~$\ref{prop:reversed}$]
Let $(g,r)\in D$. Take a $\pi_1(\Sigma_g)$-triple $(\phi_{(g,r)},Q_{(g,r)},\sigma_{(g,r)})$ for $\Gamma$ as in the definition of $\CC_{\Surf_{\{(g,r)\}}}$. From $(\sigma_{(g,r)}\colon \Gamma\twoheadrightarrow Q_{(g,r)})_{{(g,r)}\in D}$, we will construct $(Q, \sigma)$ in the following manner. Define $Q_D$ and $\sigma_D\colon \Gamma\to Q_D$ by
\[
Q_D=\prod_{(g,r)\in D}Q_{(g,r)},\quad \sigma_D\colon \gamma\mapsto (\sigma_{(g,r)}(\gamma))_{(g,r)\in D}.
\]
Set $Q:=\sigma_D(\Gamma)$ and let $\sigma\colon \Gamma\twoheadrightarrow Q$ be the surjective map defined from $\sigma_D$.
Then, by Lemma~\ref{lem:quotient}, there exists $r' = r'_{(g,r)}$ with $r \leq r' \leq 2g$ such that $\intrk^Q((\sigma \circ \phi_{(g,r)})(\pi_1(\Sigma_g))) = r'$.
Finally, set
\[
  D_{\Gamma} = \{ (g, r'_{(g,r)})\;|\; (g,r) \in D \}.
\]
We conclude that $\Gamma \in \CC_{\Surf_{D_{\Gamma}}}$.
\end{proof}

We will see examples of elements in $\CSD$ in Section~\ref{sec:surface_example}. Here, we exhibit the most basic example.
\begin{example}\label{example:surface_basic}
Let $g\in \NN$. Then we have
\[
\pi_1(\Sigma_g)\in \CC_{\Surf_{\{(g,2g)\}}}.
\]
Indeed, consider the abelianization map
\[
\Ab_{\pi_1(\Sigma_g)}\colon \pi_1(\Sigma_g)\twoheadrightarrow \pi_1(\Sigma_g)^{\ab}\simeq \ZZ^{2g}.
\]
Then $(\phi,Q,\sigma)=(\mathrm{id}_{\pi_1(\Sigma_g)},\ZZ^{2g},\Ab_{\pi_1(\Sigma_g)})$ is a $\pi_1(\Sigma_g)$-triple with $\intrk^{Q}(P)=2g$, where $P=(\sigma\circ \phi)(\pi_1(\Sigma_g))$.
\end{example}

\subsection{Proof of Theorem~\ref{thm:surface_new}}\label{subsec:proof_surface_new}
In this subsection, we prove Theorem~\ref{thm:surface_new}. The following is the key proposition. Here for a permutational wreath product $H\wr_{\rho}\Gamma$ associated with the action $\rho\colon \Gamma \curvearrowright X$, we regard an element $u\in \bigoplus_{X}H$ as a map $u\colon X\to H$ such that $u(p)=e_H$ for all but finite $p\in X$.

\begin{prop}[key proposition]\label{prop:surface_key}
Let $g\in \NN$. Let $\Gamma$ be a group and let $(\phi,Q,\sigma)$ be a $\pi_1(\Sigma_g)$-triple for $\Gamma$. Set $P=(\sigma\circ \phi)(\pi_1(\Sigma_g))$. Let $\rho_Q\colon \Gamma\curvearrowright Q$ be the $\Gamma$-action given by the composition of $\sigma\colon \Gamma\twoheadrightarrow Q$ and the action $Q\curvearrowright Q$ by multiplication. Set
\[
(G,N)=(\ZZ\wr_{\rho_Q}\Gamma, \bigoplus_{Q}\ZZ).
\]
Assume that $x\in N$ fulfills the following three conditions.
\begin{enumerate}[$(i)$]
  \item $e_{Q}\in \supp(x)$.
  \item $P$ is a zero-sum set for $x$.
  \item For every zero-sum set $S$ for $x$ satisfying $S\subseteq \supp(x)$ and $e_Q\in S$, we have $\langle S\rangle=P$.
\end{enumerate}
Then, $x\in [G,N]$ and
\[
\left\lceil \frac{\intrk^Q(P)}{2}\right\rceil\leq \cl_G(x)\leq g \quad \textrm{and} \quad \cl_{G,N}(x)=\intrk^Q(P)
\]
hold true.
\end{prop}

\begin{proof}
Recall the proof of Proposition~\ref{prop:wreath_perm}; in particular, we have  for every $v,w\in N$ and every $\gamma\in \Gamma$,
\begin{eqnarray}\label{eq:comm}
\big[ (v,\gamma), (w, e_{\Gamma})\big] = (\sigma(\gamma) w - w , e_{\Gamma}).
\end{eqnarray}
The proof of Proposition~\ref{prop:wreath_perm} shows that
\begin{eqnarray}\label{eq:int_rank_surface}
\cl_{G,N}(x)=\intrk^Q(P).
\end{eqnarray}

In what follows, we will prove that $\cl_G(x)\leq g$.
In general, we have  for every $v,w\in N$ and every $\gamma,\lambda\in \Gamma$,
\[
\big[ (v,\gamma), (w, \lambda)\big] = (\sigma(\gamma) w -\sigma([\gamma, \lambda]) w +v-\sigma(\gamma \lambda\gamma^{-1}) v, [\gamma,\lambda]).
\]
In particular,
if $\sigma(\gamma)$ commutes with $\sigma(\lambda)$, then we have
\begin{eqnarray}\label{eq:comm_comm}
\big[ (v,\gamma), (w, \lambda)\big] = (\sigma(\gamma) w - w +v-\sigma(\lambda) v, [\gamma,\lambda]);
\end{eqnarray}
Now fix a system of standard generators $(\alpha_1, \beta_1, \cdots, \alpha_g, \beta_g)$ of $\pi_1(\Sigma_g)$, meaning that,
\begin{eqnarray}\label{eq:surface_presentation}
\pi_1(\Sigma_g)=\langle \alpha_1,\beta_1,\ldots,\alpha_g,\beta_g\;|\; [\alpha_1,\beta_1]\cdots [\alpha_g,\beta_g]=e_{\pi_1(\Sigma_g)}\rangle .
\end{eqnarray}
For every $1\leq i\leq g$, set $a_i=(\sigma\circ\phi)(\alpha_i)$ and $b_i=(\sigma\circ\phi)(\beta_i)$. Then, a similar argument to the proof of Proposition~\ref{prop:fromabove} verifies the following: there exist $w_1,\ldots,w_g\in N$ and $v_1,\ldots,v_{g}\in N$ such that
\begin{eqnarray}\label{eq:surface}
x=\sum_{i=1}^g\left\{(a_iw_i-w_i)+(b_iv_i-v_i)\right\}.
\end{eqnarray}
For these $w_1,\ldots ,w_g$ and $v_1,\ldots,v_g$, set
\[
\xi_i=\big[ (-v_i,\phi(\alpha_i)), (w_i, \phi(\beta_i))\big]\ \in [G,G]
\]
for every $1\leq i\leq g$. Then, \eqref{eq:comm_comm}, \eqref{eq:surface_presentation} and \eqref{eq:surface} imply that
\[
x=\xi_1\xi_2\cdots \xi_g.
\]
Here, we observe that for every $1\leq i\leq g$, $(\sigma\circ\phi)([\alpha_i,\beta_i])=e_Q$. (Recall from the definition of $\pi_1(\Sigma_g)$-triples that $P$ is abelian.) Hence, we obtain
\begin{eqnarray}\label{eq:leq_surface}
\cl_{G}(x)\leq g,
\end{eqnarray}
as desired.
By combining \eqref{eq:int_rank_surface} and \eqref{eq:leq_surface} with Theorem~\ref{thm:split}, we obtain the conclusion.
\end{proof}

\begin{proof}[Proof of Theorem~$\ref{thm:surface_new}$]
Let $(Q,\sigma)$ be a pair that is guaranteed in the definition of $\CSD$. Fix $(g,r)\in D$. Then, there exists $\phi_{(g,r)}\colon \pi_1(\Sigma_g)\to\Gamma$ such that $(\phi_{(g,r)},Q,\sigma)$ is a $\pi_1(\Sigma_g)$-triple for $\Gamma$ satisfying \eqref{eq:ranksurf}. Take an arbitrary $x=x_{(g,r)}$  that fulfills conditions (i), (ii) and (iii) in Proposition~\ref{prop:surface_key} with respect to $\phi=\phi_{(g,r)}$. Then by Proposition~\ref{prop:surface_key}, we have
\[
\left\lceil \frac{r}{2}\right\rceil\leq \cl_G(x_{(g,r)})\leq g\quad \textrm{and}\quad \cl_{G,N}(x_{(g,r)})=r,
\]
as desired.
\end{proof}

\begin{remark}
In the setting of Theorem~\ref{thm:surface_new}, assume that $\sup\{r\,|\,(g,r)\in D\}=\infty$. Then a similar argument to one in Example~\ref{example:wreath} shows \eqref{eq:cl} and \eqref{eq:scl}.
\end{remark}

\section{Members of $\CSD$}\label{sec:surface_example}
In this section, we exhibit examples of groups in $\CSD$ and prove Theorem~\ref{thm:ex_CSD}. For a group $H$, let $H^{\ab}:=H/[H,H]$ be the abelianization of $H$, and $\Ab_H\colon H\twoheadrightarrow H^{\ab}$ be the abelianization map. Set $\NN_{\geq 2}:=\{n\in \NN\;|\; n\geq 2\}$.
\subsection{Basic examples}\label{subsec:ex}
We start from basic examples.

\begin{example}\label{example:surface}
For $g\in \NN$, take an arbitrary group quotient $\Lambda$ of $\pi_1(\Sigma_g)$ satisfying $\rk(\Lambda^{\mathrm{ab}})=r$ with $g+1\leq r\leq 2g$. Then $\Lambda\in \CC_{\Surf_{\{(g,r)\}}}$. To see this, take a quotient map $\phi\colon \pi_1(\Sigma_g)\twoheadrightarrow \Lambda$. Then, $(\phi_{(g,r)},Q,\sigma)=(\phi, \Lambda^{\mathrm{ab}}, \mathrm{Ab}_{\Lambda})$ is a $\pi_1(\Sigma_g)$-triple for $\Lambda$ that satisfies \eqref{eq:ranksurf}.
\end{example}

\begin{lem}\label{lem:freefinite}
Let $g\in \NN$. Let $(A_i)_{i=1}^g$ be a family of abelian groups of special rank at least $2$. Then the free product $\Gamma=\bigast_{i=1}^g A_i$ belongs to $\CC_{\Surf_{\{(g,2g)\}}}$.
\end{lem}

\begin{proof}
For every $1\leq i\leq g$, since $\sperk(A_i)\geq 2$, we can take a subgroup $M_i$ of $A_i$ with $\rk(M_i)=2$.
Then, the free product $\bigast_{i=1}^{g}M_i$ can be seen as a group quotient of $\pi_1(\Sigma_g)$. Here, recall presentation \eqref{eq:surface_presentation}. Let $\phi_0\colon \pi_1(\Sigma_g)\twoheadrightarrow \bigast_{i=1}^gM_i$ be the quotient map and $\iota\colon \bigast_{i=1}^gM_i\hookrightarrow \Gamma$ be the natural embedding. Then, the triple $(\phi_{(g,2g)},Q,\sigma)=(\iota\circ \phi_0, \bigoplus_{i=1}^gA_i,\Ab_{\Gamma})$ is a $\pi_1(\Sigma_g)$-triple for $\Gamma$ that satisfies \eqref{eq:ranksurf} with $r=2g$.
\end{proof}

\begin{lem}\label{lem=product}
Let $D$ be a non-empty subset of  $\{(g,r)\in \NN^2\;|\; g+1\leq r\leq 2g\}$. Let $\Lambda_{(g,r)}\in \mathcal{C}_{\mathrm{Surf}_{\{(g,r)\}}}$ for every $(g,r)\in D$. Then, three groups $\bigast_{(g,r)\in D}\Lambda_{(g,r)}$, $\bigoplus_{(g,r)\in D}\Lambda_{(g,r)}$ and $\prod_{(g,r)\in D}\Lambda_{(g,r)}$ are members of $\CSD$.
\end{lem}

\begin{proof}
Let $\Gamma$ be either of the three groups above. Fix $(g,r)\in D$. Take a $\pi_1(\Sigma_g)$-triple $(\phi_{(g,r)},Q_{(g,r)},\sigma_{(g,r)})$ for $\Lambda_{(g,r)}$  that satisfies $\intrk^{Q_{(g,r)}}(P_{(g,r)})=2g$, where $P_{(g,r)}=(\sigma_{(g,r)}\circ \phi_{(g,r)})(\pi_1(\Sigma_g))$. Let $\iota_{(g,r)}\colon \Lambda_{(g,r)} \hookrightarrow \Gamma$ be the natural embedding. Then, the triple $(\phi_{(g,r)},Q,\sigma)=(\iota_{(g,r)}\circ \phi_{(g,r)}, \Gamma^{\mathrm{ab}},\mathrm{Ab}_{\Gamma})$ is a $\pi_1(\Sigma_g)$-triple for $\Gamma$ that satisfies \eqref{eq:ranksurf}.
\end{proof}

Lemma~\ref{lem=product}, together with Example~\ref{example:surface_basic} and Lemma~\ref{lem:freefinite}, yields the following corollary.

\begin{cor}\label{cor:free}
Let $J$ be a non-empty set of $\NN$. Let $(A_i)_{i\in \NN}$ be a family of abelian groups such that $\sperk(A_i)\geq 2$ for every $i\in \NN$. Then we have
\[
\bigast_{g\in J}\pi_1(\Sigma_g)\in \CC_{\Surf_{D_J}} \quad \textrm{and} \quad \bigast_{i\in \NN}A_i\in \CC_{\Surf_{D_{\NN}}},
\]
where $D_J=\{(g,2g)\;|\; g\in J\}$.
\end{cor}

For a general group $\Gamma$ and its subgroup $\Lambda$, it seems difficult to bound $\intrk^{\Gamma}(\Lambda)$ from below. However, it is easy to check whether $\intrk^{\Gamma}(\Lambda)\leq 1$  since a group of rank at most $1$ must be cyclic. This observation yields the following proposition.

\begin{prop}\label{prop:noZ2}
Every group $\Gamma$ that contains an abelian subgroup $\Lambda$ with $\sperk(\Lambda)\geq  2$ is a member of $\mathcal{C}_{\mathrm{Surf}_{\{(1,2)\}}}$.
\end{prop}

\begin{proof}
By assumption, we can take $\Lambda_1\leqslant \Lambda$ with $\rk(\Lambda_1)=2$. Recall that $\pi_1(\Sigma_1)\simeq \ZZ^2$. Hence, there exists a surjective homomorphism $\phi\colon \pi_1(\Sigma_1)\twoheadrightarrow \Lambda_1$. Let $\iota\colon \Lambda_1\hookrightarrow \Gamma$ be the inclusion map. Then, $(\iota\circ\phi,\Gamma,\mathrm{id}_{\Gamma})$ is a $\pi_1(\Sigma_1)$-triple for $\Gamma$. We have
\[
\intrk^{\Gamma}(\Lambda_1)= 2
\]
since $\Lambda_1$ is not cyclic. Hence, $\Gamma\in \mathcal{C}_{\mathrm{Surf}_{\{(1,2)\}}}$.
\end{proof}

\subsection{Fundamental groups of mapping tori}\label{subsec:mapping}
Here we discuss  examples of groups in $\CSD$  coming from $3$-dimensional (hyperbolic) geometry. For $g\in \NN$, let $\mathrm{Mod}(\Sigma_g)$ denote the \emph{mapping class group} of $\Sigma_g$: it is defined as the group quotient of the group of orientation-preserving diffeomorphisms modulo isotopy. It is well known that reduction to the action on $\HHH_1(\Sigma_g;\ZZ)\simeq \ZZ^{2g}$ (equipped with a natural symplectic structure coming from the intersection form) produces the natural symplectic representation $s_g\colon \mathrm{Mod}(\Sigma_g)\twoheadrightarrow \mathrm{Sp}(2g,\ZZ)$ of $\mathrm{Mod}(\Sigma_g)$. For an orientation-preserving diffeomorphism $f\colon \Sigma_{g} \to \Sigma_{g}$, let $T_f$ denote the \emph{mapping torus} of $f$, meaning that,
\[
T_f:=(\Sigma_g\times [0,1])/((p,0)\sim(f(p),1)\ \textrm{for every}\ p\in \Sigma_g).
\]
The celebrated theorem by Thurston \cite{Thurston} states that $T_f$ is a hyperbolic manifold if and only if $\psi\in \mathrm{Mod}(\Sigma_g)$ is a pseudo-Anosov class if and only if $\pi_1(T_f)$ does not contain $\ZZ^2$. Hence, if $[f]$ is not a pseudo-Anosov class, then $\pi_1(T_f)\in \mathcal{C}_{\mathrm{Surf}_{\{(1,2)\}}}$ by Proposition~\ref{prop:noZ2}.

The fundamental group $\pi_1(T_f)$ is described in terms of the isotopy class $\psi=[f]\in \Mod(\Sigma_g)$ as follows. Let $\Psi\in \mathrm{Aut}(\pi_1(\Sigma_g))$ be the automorphism of $\pi_1(\Sigma_g)$ induced by $f$. Then we have a natural isomorphism
\begin{eqnarray}\label{eq:mapping_pi1}
\pi_1(T_f)\simeq \pi_1(\Sigma_g)\rtimes_{\Psi}\ZZ.
\end{eqnarray}
Here in the right hand side, the $\ZZ$-action is given by $\Psi$. Then, formation of the abelianization of $\pi_1(\Sigma_g)$ induces the quotient map
\begin{eqnarray}\label{eq:mapping_solv}
\sigma\colon \pi_1(T_f)\twoheadrightarrow \ZZ^{2g}\rtimes_{s_g(\psi)}\ZZ,
\end{eqnarray}
where the $\ZZ$-action on $\ZZ^{2g}$ is given by $s_g(\psi)\in \Sp(2g,\ZZ)$. Therefore, we have the following result.

\begin{lem}\label{lem:mapping}
Let $g\in \NN_{\geq 2}$ and let $\psi\in \Mod(\Sigma_g)$. Let $f\colon \Sigma_g\to\Sigma_g$ be a diffeomorphism on $\Sigma_g$ whose isotopy class is $\psi$. Let $\phi\colon \pi_1(\Sigma_g)\hookrightarrow \pi_1(T_f)$ be the natural embedding from \eqref{eq:mapping_pi1}. Let $\sigma$ be the map in \eqref{eq:mapping_solv}. Then, the triple $(\phi,Q,\sigma)=(\phi, \ZZ^{2g}\rtimes_{s_g(\psi)}\ZZ,\sigma)$ is a $\pi_1(\Sigma_g)$-triple for $\pi_1(T_f)$.
\end{lem}

In Lemma~\ref{lem:mapping}, set $Q=\ZZ^{2g}\rtimes_{s_g(\psi)}\ZZ$ and $P=(\sigma\circ \phi)(\pi_1(\Sigma_g))(\simeq \ZZ^{2g})$. The next task is to compute $\intrk^Q(P)$ from below. Levitt--Metaftsis \cite{LM} and Amoroso--Zannier \cite{AZ} obtained the following result. Here for $d\in \NN$, let $\mathrm{Mat}_{d\times d}(\ZZ)$ denote the ring of $d\times d$ integer matrices; we regard it as a subring of $\mathrm{Mat}_{d\times d}(\mathbb{C})$, the ring of $d\times d$ complex  matrices and discuss eigenvalues and eigenspaces of elements of $\mathrm{Mat}_{d\times d}(\ZZ)$ as those of $\mathrm{Mat}_{d\times d}(\mathbb{C})$. Their results are stated in terms of the following concepts.

\begin{definition}
Let $d\in \NN_{\geq 2}$ and $A\in \mathrm{Mat}_{d\times d}(\ZZ)$.
\begin{enumerate}[$(1)$]
    \item Let $v\in \ZZ^d$. Then, the \emph{$A$-orbit} of $v$ is the set
\[
\{A^nv\;|\; n\in\ZZ_{\geq 0}\},
\]
where $A^0:=I_d$.
  \item We define $\OR(A)$ as the minimal number of elements in $\ZZ^d$ whose $A$-orbits generate $\ZZ^d$ as a $\ZZ$-module.
\end{enumerate}
\end{definition}

If $A\in \GL(n,\ZZ)$, then the Cayley--Hamilton theorem implies that the $\ZZ$-module generated by the $A$-orbit of $v$ equals that generated by the set $\{A^nv\;|\; n\in\ZZ\}$ for every $v\in \ZZ^d$.

\begin{thm}[\cite{LM}, \cite{AZ}]\label{thm:rkmapping}
Let $d\in \NN_{\geq 2}$ and $A\in \mathrm{Mat}_{d\times d}(\ZZ)$. Set
\begin{eqnarray}\label{eq:numbertheory}
C_d=\prod_{q\leq d}q,
\end{eqnarray}
where $q$ runs over the prime powers at most $d$.
Then the following hold.
\begin{enumerate}[$(1)$]
  \item $($\cite[Corollary~2.4]{LM}$)$ Assume that $A\in \GL(d,\ZZ)$. Let $H=\ZZ^d\rtimes_A\ZZ$, where $\ZZ$-action on $\ZZ^d$ is given by $A$. Then, we have
\[
\rk(H)=1+\OR(A).
\]
  \item $($\cite[Theorem~1.5]{LM}$)$ Assume that $A\in\GL(d,\ZZ)$ and that $A$ is of infinite order. Then, there exists $n_0\in \NN$ such that for every $n\in \NN$ with $n\geq n_0$,
\[
\OR(A^n)\geq 2
\]
holds.
  \item $($\cite[Theorem~1.5]{AZ}$)$ There exists an effective absolute constant $c>0$ such that in the setting of $(2)$, we can take
\[
n_0=\lceil cd^6(\log d)^6\rceil.
\]
 \item $($\cite[Remark~4.1]{AZ}$)$ Assume that $A$ has only one eigenvalue. Then for every $n\geq C_d$, we have
\[
\OR(A^n)= d.
\]
 \item $($\cite[Remark~4.2]{AZ}$)$ Assume that $A$ has two eigenvalues whose ratio is not a root of unity. Let $r$ be the sum of the dimensions of their eigenspaces. Then for every $n\geq C_d$, we have
\[
\OR(A^n)\geq  r.
\]
\end{enumerate}
\end{thm}

In~(4), if $A\in \GL(d,\ZZ)$, then the unique eigenvalue of $A$ must be either $1$ or $-1$. We state the following immediate corollary to Theorem~\ref{thm:rkmapping}~(1).

\begin{cor}\label{cor:intrank}
Let $d\in \NN_{\geq 2}$ and $A\in \GL(d,\ZZ)$. Let $H=\ZZ^d\rtimes_A\ZZ$ and $K=\ZZ^d$, the kernel of the natural projection $H\twoheadrightarrow \ZZ$. Then
\[
\intrk^H(K)=\min\{d, \rk(H)\}=\min\{d, 1+\OR(A)\}.
\]
\end{cor}

\begin{proof}
First observe that every group $\Theta$ with $K\leqslant \Theta \leqslant H$ is of the form $K\rtimes_{A}(l\ZZ)$ with $l\in \ZZ_{\geq 0}$. By Theorem~\ref{thm:rkmapping}~(1), for every $l\in \NN$, we have
\[
\rk(K\rtimes_{A}(l\ZZ))=1+\OR(A^l)\geq 1+\OR(A)=\rk(K\rtimes_{A}\ZZ).
\]
If $l=0$, then $\rk(K\rtimes_{A}(0\ZZ))=\rk(K)=d$. Hence, we obtain the conclusion.
\end{proof}

By letting $d=2g$, we have the following proposition from Lemma~\ref{lem:mapping} and Corollary~\ref{cor:intrank}.

\begin{prop}\label{prop:mapping_torus}
Let $g\in \NN_{\geq 2}$. Let  $\psi\in \mathrm{Mod}(\Sigma_g)$. Let $f \colon \Sigma_{g} \to \Sigma_{g}$ be a diffeomorphism whose isotopy class $[f]$ is $\psi$. Let $s_g\colon \mathrm{Mod}(\Sigma_g)\twoheadrightarrow \mathrm{Sp}(2g,\ZZ)$ be the symplectic representation. Let $Q=\ZZ^{2g}\rtimes_{s_g(\psi)} \ZZ$. Assume that  $\rk(Q)\geq g+1$, equivalently, that $\OR(s_g(\psi))\geq g$. Then,
\[
\pi_1(T_f)\in \CC_{\Surf_{\{(g,r)\}}},
\]
where $r=\min\{2g, \OR(s_g(\psi))+1\}$.
\end{prop}

Then, (2)--(5) of Theorem~\ref{thm:rkmapping} yield the following theorem. We recall that the kernel of the symplectic representation $s_g\colon \Mod(\Sigma_g)\twoheadrightarrow \Sp(2g,\ZZ)$ is called the \emph{Torelli group} $\mathcal{I}(\Sigma_g)$ of $\Sigma_g$.

\begin{thm}[groups in $\CSD$ from mapping tori]\label{thm:mapping_tori}
Assume the setting of Proposition~$\ref{prop:mapping_torus}$. Then the following hold true. Here, the constant $C_d$ for $d\in \NN_{\geq 2}$ is given by \eqref{eq:numbertheory}.
\begin{enumerate}[$(1)$]
  \item Assume that $s_g(\psi)\in \{\pm I_{2g}\}$. Then, $\pi_1(T_f)\in \CC_{\Surf_{\{(g,2g)\}}}$ holds. In particular, if $\psi\in \mathcal{I}(\Sigma_g)$, then we have $\pi_1(T_f)\in \CC_{\Surf_{\{(g,2g)\}}}$.
  \item Assume that $s_g(\psi)\in \Sp(2g,\ZZ)$ has only one eigenvalue $($hence either $1$ or $-1$$)$. Then, for every $n\in \NN$ with $n\geq C_{2g}$, we have
\[
\pi_1(T_{f^n})\in \CC_{\Surf_{\{(g,2g)\}}}.
\]
  \item Let $t$ be an integer with $t\geq g$. Assume that $s_g(\psi)$ has two eigenvalues whose ratio is not a root of unity. Moreover, assume that the sum of the dimensions of their eigenspaces is at least $t$. Then for every $n\in \NN$ with $n\geq C_{2g}$, we have
\[
\pi_1(T_{f^n})\in \bigcup_{r=\min\{2g,t+1\}}^{2g}\CC_{\Surf_{\{(g,r)\}}}.
\]
  \item Assume that $g=2$. Then there exists an effective absolute constant $n_0\in \NN$ such that the following holds true: assume that $s_2(\psi)$ is of infinite order. Then for every $n\in \NN$ with $n\geq n_0$, we have
\[
\pi_1(T_{f^n})\in \CC_{\Surf_{\{(2,3)\}}}\cup \CC_{\Surf_{\{(2,4)\}}}.
\]
\end{enumerate}
\end{thm}

\begin{proof}
We apply Proposition~\ref{prop:mapping_torus}. Item (1) follows because $\OR(I_{2g})=\OR(-I_{2g})=2g$. (In this case, we can also determine the intermediate rank directly.) Item (2) follows from Theorem~\ref{thm:rkmapping}~(4); (3) follows from Theorem~\ref{thm:rkmapping}~(5). Finally (4) can be derived from (2) and (3) of Theorem~\ref{thm:rkmapping}. Indeed, take the effective absolute constant $c>0$ as in Theorem~\ref{thm:rkmapping}~(3) and   set
\[
n_0=\lceil c\cdot  4^6(\log 4)^6\rceil.
\]
Then for every $n\in \NN$ with $n\geq n_0$, we have
\[
\rk(\ZZ^{2g}\rtimes_{s_g(\psi)^n}\ZZ)=1+\OR(s_g(\psi)^n)\geq 3;
\]
hence we obtain the conclusion.
\end{proof}

We are now ready to prove Theorem~\ref{thm:ex_CSD}.
\begin{proof}[Proof of Theorem~$\ref{thm:ex_CSD}$]
Item (1) is stated as  Proposition~\ref{prop:noZ2} and Theorem~\ref{thm:mapping_tori}~(5); (2) follows from Example~\ref{example:surface_basic} and Corollary~\ref{cor:free}. Items (3) and (4) follow from (4) and (1) of Theorem~\ref{thm:mapping_tori}, respectively.
\end{proof}

\begin{remark}\label{rem:residual}
If $J\subseteq \NN_{\geq 2}$, then the group $\Gamma=\bigast_{g\in J}\pi_1(\Sigma_g)$ is \emph{fully residually free}, that means, for every finite subset $S$ of $\Gamma$, there exist a free group $H$ and a homomorphism $\phi\colon \Gamma\to H$ such that $\phi|_S$ is injective. To see this, first recall from \cite{GBaum} that $\pi_1(\Sigma_g)$ is fully residually free for every $g\in \NN_{\geq 2}$. Then, apply results in \cite{BBaum}.
\end{remark}

\begin{remark}\label{rem:souto}
Souto \cite{Souto} showed the following: if $\psi\in \Mod(\Sigma_g)$ is a pseudo-Anosov mapping class in the setting of Proposition~\ref{prop:mapping_torus}, then there exists $n_{\psi}\in \NN$ such that for every $n\in \NN$ with $n\geq n_{\psi}$, we have $\rk(\pi_1(T_{f^n}))=2g+1$. Hence for every $n\in \NN$ with $n\geq n_{\psi}$, we have
\[
\intrk^{\pi_1(T_{f^n})}(\pi_1(\Sigma_g))=2g.
\]
\end{remark}

We finally pose the following problems, which seem open, in relation to Theorems~\ref{thm local cyclicity}, \ref{thm:abelian} and \ref{thm:surface_new}.

\begin{problem}\label{prob:exist}
\begin{enumerate}[$(1)$]
  \item Does there exist a non-abelian group $\Gamma$ such that for every pair $(G,N)$ fitting in $(\star)$, $\cl_G$ and $\cl_{G,N}$ coincide on $[G,N]$?
  \item Does there exist a non-abelian group $\Gamma$ such that for every pair $(G,N)$ fitting in \emph{split} short exact sequence $(\star)$, $\cl_G$ and $\cl_{G,N}$ coincide on $[G,N]$?
  \item Find a `good' class $\mathcal{C}$ of groups $\Gamma$ such that for every pair $(G,N)$ fitting in $(\star)$,
\[
\sup_{x\in [G,N]}(\cl_{G,N}(x)-\cl_{G}(x))<\infty
\]
holds.
 \item Find a `good' class $\mathcal{C}$ of groups $\Gamma$ such that for every pair $(G,N)$ fitting in \emph{split} short exact sequence $(\star)$,
\[
\sup_{x\in [G,N]}(\cl_{G,N}(x)-\cl_{G}(x))<\infty
\]
holds.
\end{enumerate}
\end{problem}

\section{Concluding remarks}\label{sec:remark}
\subsection{Examples from symplectic geometry}\label{subsec:symp}
In this subsection, we exhibit examples of triples $(G,N,\Gamma)$ that fit in $(\star)$ from symplectic geometry. In the first example, $\cl_G$ coincides with $\cl_{G,N}$ on $[G,N]$ and $G/N\simeq \RR$.
For basic concepts of symplectic geometry, see \cite{HZ}, \cite{MS}, and \cite{PR}.

A symplectic manifold is  said to be \textit{exact} if the symplectic form is  an exact form.
\begin{prop}\label{ham cal}
Let $(M,\omega)$  be an exact symplectic manifold.
Set $G=\Ham(M,\omega)$, where $\Ham(M,\omega)$ is the group of Hamiltonian diffeomorphisms $($with compact support$)$ of $(M,\omega)$ and $N$ the commutator subgroup $[G,G]$ of $G$.
Then,  the following hold true.
\begin{enumerate}[$(1)$]
\item $N=[G,N]$.
\item $G/N\simeq \RR$.
\item $\cl_G$ and $\cl_{G,N}$ coincide on $N$.
\end{enumerate}
\end{prop}

Here, we remark that (1) and (2) are known.  To prove Proposition \ref{ham cal}, we use the Calabi homomorphism.
For an exact symplectic manifold $(M,\omega)$, we recall that the \textit{Calabi homomorphism}
is a function $\mathrm{Cal} \colon \Ham(M,\omega)\to\mathbb{R}$ defined by
\[
	\mathrm{Cal}(\varphi_H)=\int_0^1\int_M H_t\omega^n\,dt,
\]
 where $\varphi_H$ is the Hamiltonian diffeomorphism generated by a smooth function $H \colon [0,1] \times M \to \RR$.
 It is known that the Calabi homomorphism is a well-defined  surjective group homomorphism and  that $\Ker(\Cal)=N$ (see \cite{Cala,Ban,Ban97,MS,Hum}). We also remark that $N$ is perfect (see \cite{Ban} and \cite{Ban97}).

We also note that every exact symplectic manifold is open.
Indeed, it is known that the symplectic form of a closed symplectic manifold is cohomologically non-trivial (see \cite[Section 1.1]{HZ}  and  \cite{MS}).

\begin{proof}[Proof of Proposition~$\ref{ham cal}$]
First, we prove (1). As mentioned earlier, the group $N$ is known to be perfect. Hence,
we have  $N=[N,N]\leqslant [G,N]$.
Since $N$ is a normal subgroup of $G$, we have  $[G,N]\leqslant N$. Therefore, we conclude that
\[
N=[G,N].
\]
Item (2) holds since the Calabi homomorphism is surjective and $\Ker(\Cal)=N$.

Finally, we prove (3). Let $f,g\in G$.
In what follows, we will show that $[f,g]$ is a $(G,N)$-commutator. As we mentioned above, every exact symplectic manifold is open. Hence, we can take $h\in \hG$ such that the following two conditions are fulfilled:
\begin{itemize}
  \item the support of $h$ is disjoint from that of $f$;
  \item $\mathrm{Cal}(h)=-\mathrm{Cal}(g)$.
\end{itemize}
By the first condition, $[f,g]=[f,gh]$  holds.
By the second condition,  we have
\[
\Cal(gh)=\Cal(g)+\Cal(h)=0;
\]
it implies that $gh\in N$ since $\Ker(\Cal)=N$. Therefore, every $(G,G)$-commutator is a $(G,N)$-commutator, and hence $\cl_G$ coincides with $\cl_{G,N}$ on $[G,N]=N$. This completes our proof.
\end{proof}

In the proof of Proposition~\ref{ham cal}, we use the various properties of the Calabi homomorphism.
If we consider the analogue of Proposition~\ref{ham cal} on the flux homomorphism, then the following problem seems open.

\begin{problem}\label{flux}
Let $(M,\omega)$  be a closed symplectic manifold with $\HHH^1(M;\RR)=\RR$.
Let $G$ be the identity component $\Symp_0(M,\omega)$ of the group $\Symp(M,\omega)$ of symplectomorphisms of $(M,\omega)$ and $N$ the group of Hamiltonian diffeomorphisms of $(M,\omega)$.

Then, does $\cl_G$ and $\cl_{G,N}$ coincide on $N$?
\end{problem}

We note that under the setting of Problem \ref{flux}, there exists a subgroup $\Gamma_{\omega}$ of $\HHH^1(M;\RR)$, which is called the \textit{flux group} of $(M,\omega)$ such that $G/N=\HHH^1(M;\RR)/\Gamma_{\omega}$ (\cite{Ban}, \cite{Ban97}).
We also note that $\Gamma_{\omega}$ is known to be always discrete in $\HHH^1(M;\RR)$ (\cite{O}) and that the quotient homomorphism $G\to G/N$ has a section homomorphism if $\HHH^1(M;\RR)=\RR$.

For examples of closed symplectic manifolds with $\HHH^1(M;\RR)=\RR$, see \cite{G95} and \cite{HAP}.

The second example comes from the following proposition and corollary.

\begin{prop}\label{prop:cw}
Let $\Gamma$ be a group.
Assume that the commutator width of $\Gamma$ is finite, meaning that
\[
\sup_{\gamma \in[\Gamma,\Gamma]}\cl_{\Gamma}(\gamma)<\infty.
\]
Let $n_{\Gamma}\in \ZZ_{\geq 0}$ be the quantity defined on the left hand side. Set
\begin{eqnarray}\label{eq:ab_perm}
(G,N)=(\ZZ\wr_{\rho_{\Gamma^{\ab}}}\Gamma,\bigoplus_{\Gamma^{\ab}}\ZZ),
\end{eqnarray}
where $\rho_{\Gamma^{\ab}}\colon \Gamma\curvearrowright \Gamma^{\ab}$ is the composition of $\Ab_{\Gamma}$ and $\Gamma^{\ab}\curvearrowright \Gamma^{\ab}$ by left multiplication. Then for every $x\in [G,N]$, we have
\[
\left\lceil \frac{\cl_{G,N}(x)}{2}\right\rceil\leq  \cl_{G}(x)\leq \left\lceil \frac{\cl_{G,N}(x)}{2}\right\rceil+n_{\Gamma}.
\]
\end{prop}

\begin{cor}\label{cor:cw}
Let $\Gamma$ be a group. Assume that
\[
n_{\Gamma}=\sup_{\gamma \in[\Gamma,\Gamma]}\cl_{\Gamma}(\gamma)<\infty \quad \textrm{and}\quad \sperk(\Gamma^{\ab})=\infty.
\]
Let $(G,N)$ be the pair defined by $\eqref{eq:ab_perm}$. Then, we have
\[
\sup_{x\in [G,N]}(\cl_{G,N}(x)-C\cdot \cl_{G}(x))=\infty
\]
for every real number $C < 2$ but
\[
\sup_{x\in [G,N]}(2\cl_{G}(x)-\cl_{G,N}(x))\leq 2n_{\Gamma}+1.
\]
\end{cor}
 We note that Propositions~\ref{prop:wreath_perm} and \ref{prop:cw} recover Theorem~\ref{thm:abelian}. Indeed, if $\Gamma$ is abelian, then $n_{\Gamma}=0$ and $\Gamma^{\ab}=\Gamma$ hold.

Symplectic geometry supplies the following interesting example to which Corollary~\ref{cor:cw} applies.

\begin{example}\label{ex:symp}
Let $(M,\omega)$ be an exact symplectic manifold.
Set $\Gamma=\Ham(M\times \RR^{2},\mathrm{pr_1}^\ast\omega+\mathrm{pr_2}^\ast\omega_{0})$, where $\mathrm{pr_1}\colon M\times\RR^2 \to M$, $\mathrm{pr_2}\colon M\times\RR^2 \to \RR^2$ are the first, second projection, respectively and $\omega_{0}$ is the standard symplectic form on $\RR^{2}$.
Then this $\Gamma$ satisfies that
\begin{equation}\label{r2n ham}
\sup_{\gamma\in [\Gamma,\Gamma]}\cl_{\Gamma}(\gamma)\leq 2 \quad \textrm{and}\quad \sperk(\Gamma^{\ab})=\infty.
\end{equation}
Indeed, the former assertion follows from the work of Burago--Ivanov--Polterovich \cite[Corollary~2.3]{BIP}; the latter holds since $\Gamma^{\ab}\simeq \RR$.
Here, note that $(M\times \RR^{2},\mathrm{pr_1}^\ast\omega+\mathrm{pr_2}^\ast\omega_{0})$ is an exact symplectic manifold and hence Proposition \ref{ham cal}~(2) applies.

In particular, $\Gamma=\Ham(\RR^{2n},\omega_{0})$ for every $n\geq 1$ satisfies \eqref{r2n ham}. Here, $\omega_{0}$ is the standard symplectic form on $\RR^{2n}$.
\end{example}

\begin{proof}[Proofs of Proposition~$\ref{prop:cw}$ and Corollary~$\ref{cor:cw}$]
First, we prove Proposition~\ref{prop:cw}. Let $x\in [G,N]$ be a non-trivial element and set $r=\cl_{G,N}(x)\in \NN$. Then, by \eqref{eq:comm}, there exist $\gamma_1,\ldots ,\gamma_r\in \Gamma$ and $w_1,\ldots ,w_r\in N$ such that
\[
x=\sum_{i=1}^r (\Ab_{\Gamma}(\gamma_i)w_i-w_i).
\]
First, we treat the case where $r$ is even. Since $\Gamma^{\ab}$ is abelian, we then have
\[
(x,\xi)=[(-w_2,\gamma_1),(w_1,\gamma_2)][(-w_4,\gamma_3),(w_3,\gamma_4)]\cdots [(-w_r,\gamma_{r-1}),(w_{r-1},\gamma_{r})],
\]
where $\xi$ is defined by
\[
\xi=[\gamma_1,\gamma_2][\gamma_3,\gamma_4]\cdots[\gamma_{r-1},\gamma_r]\in [\Gamma,\Gamma].
\]
By assumption, $\xi^{-1}$ may be written as the product of at most $n_{\Gamma}$ single commutators. This means, there exist an integer $k\leq  n_{\Gamma} $, elements $\lambda_1,\ldots ,\lambda_k\in \Gamma$ and $\lambda'_1,\ldots ,\lambda'_k\in \Gamma$ such that
\[
\xi^{-1}=[\lambda_1,\lambda'_1][\lambda_2,\lambda'_2]\cdots[\lambda_k,\lambda'_k].
\]
Therefore, we have
\[
x=[(-w_2,\gamma_1),(w_1,\gamma_2)]\cdots [(-w_r,\gamma_{r-1}),(w_{r-1},\gamma_{r})][(0,\lambda_1),(0,\lambda'_1)]\cdots[(0,\lambda_k),(0,\lambda'_k)]
\]
and
\[
\cl_{G}(x)\leq \frac{r}{2}+k\leq \frac{r}{2}+n_{\Gamma}.
\]
For the case where $r$ is odd, a similar argument to one above shows that
\[
\cl_{G}(x)\leq \frac{r+1}{2}+n_{\Gamma}.
\]
By combining these two inequalities with Theorem~\ref{thm:split}, we obtain the conclusion of Proposition~\ref{prop:cw}.

Finally, we prove Corollary~\ref{cor:cw}: it immediately follows from Proposition~\ref{prop:wreath_perm}, Proposition~\ref{prop:cw} and Lemma~\ref{lem:abel}.
\end{proof}

\subsection{Examples of groups of finite general rank}\label{subsec:gen_rk}
For a group $\Gamma$ that may not be finitely generated, we have two notions of ranks due to Malcev: the \emph{general rank} $\genrk(\Gamma)$ (Definition~\ref{def:int_rk}) and the \emph{special rank} $\sperk(\Gamma)$ (Definition~\ref{def:local_rank}). For abelian $\Gamma$, these two coincide (Lemma~\ref{lem:abel}). However, $\genrk(\Gamma)$ can be much smaller than $\sperk(\Gamma)$ in general. For instance, $\Gamma=F_2$, we have
\[
\genrk(F_2)=2\quad \textrm{and}\quad \sperk(F_2)=\infty.
\]
Here for $n\in \NN$, $F_n$ denotes the free group of rank $n$. To see the latter, we observe that $F_n$ embeds into $F_2$ for every $n\in \NN$ (to see the former, see Example~\ref{ex:fg}). Here we list basic properties of general ranks and exhibit some examples of groups of finite general rank. The contents in this subsection might be known to the experts on general ranks; nevertheless, we include the proofs for the sake of convenience. We refer the reader to \cite{Azarov17} for study of general ranks.

\begin{example}\label{ex:spe}
By definition, we have  $\genrk(\Gamma)\leq \sperk(\Gamma)$ for every group $\Gamma$. Hence every group $\Gamma$ of finite special rank is of finite general rank. Groups of finite special rank have been studied by various researchers; see \cite{DKS}.
\end{example}

\begin{example}\label{ex:fg}
Let $\Gamma$ be a finitely generated group. Then we have
\[
\genrk(\Gamma)=\rk(\Gamma)<\infty.
\]
To see this, note that $\intrk^{\Gamma}(\Gamma)=\rk(\Gamma)$ and $\intrk^{\Gamma}(\Lambda)\leq \rk(\Gamma)$ for every subgroup $\Lambda$ of $\Gamma$.
\end{example}

\begin{remark}\label{rem:cyc}
We recall from the proof of Proposition~\ref{prop:noZ2} that for a group $\Gamma$ and its finitely generated subgroup $\Lambda$, $\intrk^{\Gamma}(\Lambda)\leq 1$ holds if and only if $\Lambda$ is cyclic. Therefore, $\Gamma$ is of general rank at least $1$ if and only if $\Gamma$ is locally cyclic.
\end{remark}

From the viewpoints of Examples~\ref{ex:spe} and \ref{ex:fg}, in what follows, we seek for groups of finite general ranks such that they are of infinite special rank and that they are not finitely generated. First, we discuss some permanence properties of having finite general ranks.

\begin{prop}[stability under  taking group quotients]\label{prop:quot}
Let $\Gamma$ be a group and $Q$ be a group quotient of $\Gamma$. Then we have
\[
\genrk(\Gamma)\geq \genrk(Q).
\]
In particular, we have
\[
\genrk(\Gamma)\geq \sperk(\Gamma^{\ab}).
\]
\end{prop}

\begin{prop}[stability under  inductive limits]\label{prop:ind}
Let $(\Gamma_i,\iota_{ij})$ be an injective system of groups, namely, $\iota_{ij}\colon \Gamma_i\hookrightarrow \Gamma_{j}$ is an injective group homomorphism for every $i,j\in \NN$ with $i\leq j$ such that $\iota_{ii}=\mathrm{id}_{\Gamma_i}$ for every $i\in \NN$ and $\iota_{jk}\circ \iota_{ij}=\iota_{ik}$ for every $i,j,k\in\NN$ with $i\leq j\leq k$. Let $\Gamma=\varinjlim \Gamma_i$ be the inductive limit of $(\Gamma_i,\iota_{ij})$. Then we have
\[
\genrk(\Gamma)\leq \liminf_{i\to \infty} \genrk(\Gamma_i).
\]

In particular, if $\Gamma_1\leqslant\Gamma_2\leqslant\Gamma_3\cdots\leqslant \Gamma_i\leqslant \Gamma_{i+1}\leqslant \cdots$ is an increasing sequence of groups, then we have
\begin{eqnarray}\label{eq:ind_ineq}
\genrk(\bigcup_{i\in \NN}\Gamma_i)\leq \liminf_{i\to \infty} \genrk(\Gamma_i).
\end{eqnarray}
\end{prop}

In general, the equality does \emph{not} hold in $\eqref{eq:ind_ineq}$; see Example~\ref{example not equal}.

\begin{prop}[stability under extensions]\label{prop:ext_rk}
Assume that
\[
1 \longrightarrow N \longrightarrow G \longrightarrow \Gamma \longrightarrow 1
\]
is a short exact sequence of groups. Then we have
\begin{eqnarray}\label{eq:loc_intrk_ineq}
\genrk(G)\leq \genrk(N)+\genrk(\Gamma).
\end{eqnarray}
\end{prop}

\begin{prop}[stability under taking overgroups and subgroups of finite indices]\label{prop:subgrp}
Let $\Gamma$ be a group and $\Lambda$ its subgroup. Assume that $\Gamma$ is non-trivial. Then we have
\begin{eqnarray}\label{eq:over}
\genrk(\Gamma)\leq \genrk(\Lambda)+[\Gamma:\Lambda]-1
\end{eqnarray}
and that
\begin{eqnarray}\label{eq:sub}
\genrk(\Lambda)\leq [\Gamma:\Lambda]\cdot (\genrk(\Gamma)-1)+1.
\end{eqnarray}
Here, $[\Gamma:\Lambda]$ denotes $\#(\Gamma/\Lambda)$, the \emph{index} of $\Lambda$ in $\Gamma$.
\end{prop}

In Proposition~\ref{prop:subgrp}, if $\Lambda$ is normal in $\Gamma$, then \eqref{eq:loc_intrk_ineq} provides a better bound
\[
\genrk(\Gamma)\leq \genrk(\Lambda)+\genrk(\Gamma/\Lambda)
\]
than \eqref{eq:over}.

\begin{prop}[stability under wreath products]\label{prop:wreath_rk}
Let $H$ and $\Gamma$ be groups. Then we have
\begin{eqnarray}\label{eq:loc_intrk_ineq_wreath}
\genrk(H\wr \Gamma)\leq \genrk(H)+\genrk(\Gamma).
\end{eqnarray}
\end{prop}

\begin{proof}[Proof of Proposition~$\ref{prop:quot}$]
The former assertion immediately follows from Lemma~\ref{lem:quotient}; the latter assertion then holds by Lemma~\ref{lem:abel}.
\end{proof}

\begin{proof}[Proof of Proposition~$\ref{prop:ind}$]
This follows from the very definition of general rank. Indeed, every finite subset $S$ of $\Gamma$ may be regarded as a subset of $\Gamma_i$ for a sufficiently large $i$, depending on $S$.
\end{proof}

\begin{proof}[Proof of Proposition~$\ref{prop:ext_rk}$]
For the proof, we may assume that $\genrk(N)$ and $\genrk(\Gamma)$ are finite. Let $q\colon G\twoheadrightarrow \Gamma$ be the quotient map in the short exact sequence. Set $m=\genrk(N)$ and $l=\genrk(\Gamma)$. Take finitely many elements $g_1,\ldots,g_k\in G$ arbitrarily.
Since $\genrk(\Gamma)=l$, there exists a subgroup $\Theta$ of $\Gamma$ such that
\[
\langle q(g_1),q(g_2),\ldots ,q(g_k)\rangle \leqslant \Theta \leqslant \Gamma \quad \textrm{and}\quad \rk(\Theta)\leq l.
\]
Set $s=\rk(\Theta)$, and take a generating set $\{\theta_1,\ldots ,\theta_s\}$ of $\Theta$ of size $s$. For every $1\leq i\leq s$, fix  $h_i\in G$ satisfying $q(h_i)=\theta_i$. Set $H=\langle h_1,\ldots ,h_s\rangle$. Then, by construction, there exist $f_1,\ldots ,f_k\in H$ such that for every $1\leq j\leq k$, $q(g_jf_j^{-1})=e_{\Lambda}$ holds. This is equivalent to saying that $g_jf_j^{-1}\in N$. For every $1\leq j\leq k$, set $x_j=g_jf_j^{-1}\in N$. Then, since $\genrk(N)=m$, there exists a subgroup $K$ of $N$ with
\[
\langle x_1,x_2,\ldots ,x_k\rangle \leqslant K \leqslant N
\]
such that $\rk(K)\leq m$. For such $K$, we have
\[
\langle g_1,\ldots ,g_k\rangle \leqslant \langle K\cup H\rangle \leqslant G \quad \textrm{and}\quad \rk(\langle K\cup H\rangle )\leq \rk(K)+s\leq m+l.
\]
This implies \eqref{eq:loc_intrk_ineq}, as desired.
\end{proof}

\begin{proof}[Proof of Proposition~$\ref{prop:subgrp}$]
First we prove \eqref{eq:over}. We may assume that $\genrk(\Lambda)$ and $[\Gamma:\Lambda]$ are finite. Set $m=\genrk(\Lambda)$ and $l=[\Gamma:\Lambda]$. Take a set $\{s_1=e_{\Gamma},\ldots ,s_l\}$ of complete representatives of $\Gamma/\Lambda$. Take an arbitrary finitely generated subgroup $\Xi$ of $\Gamma$; let $\{\xi_1,\ldots ,\xi_k\}$ be a set of generators of size $k$, where $k=\rk(\Xi)$. Then, for every $1\leq i\leq k$, there exists a unique $1\leq j_i\leq l$ such that $s_{j_i}^{-1}\xi_i \in \Lambda$. Let $H$ be the subgroup of $\Lambda$ generated by $\{s_{j_i}^{-1}\xi_i\;|\; 1\leq i\leq k\}$. Since $\genrk(\Lambda)=m$, there exists a subgroup $\Theta$ of $\Lambda$ such that
\[
H\leqslant \Theta\leqslant \Lambda \quad \textrm{and} \quad \rk(\Theta)\leq m.
\]
We then observe that
\[
\Xi\leqslant \langle \Theta\cup \{s_2,\ldots ,s_l\}\rangle \leqslant \Gamma;
\]
hence obtaining $\genrk(\Gamma)\leq m+l-1$.

Next, we show \eqref{eq:sub}. Before proceeding to the proof, we recall the following variant of \emph{Schreier's subgroup lemma}: let $H$ be a non-trivial finitely generated group and $K$ be a subgroup of finite index in $H$. Then
\begin{eqnarray}\label{eq:index}
\rk(K)\leq [H:K]\cdot (\rk(H)-1)+1
\end{eqnarray}
holds. For instance, see \cite[Proposition~12.1 in Chapter~III]{LSbook}.

Again, we may assume that $\genrk(\Lambda)$ and $[\Gamma:\Lambda]$ are finite. Set $m=\genrk(\Gamma)$ and $l=[\Gamma:\Lambda]$. Since $\Gamma$ is non-trivial, we have $m\geq 1$. Take an arbitrary finitely generated non-trivial subgroup $\Xi$ of $\Lambda$. Then since $\genrk(\Gamma)=m$, there exists a non-trivial subgroup $\Theta$ of $\Gamma$ such that
\[
\Xi\leqslant \Theta\leqslant \Gamma \quad \textrm{and} \quad \rk(\Theta)\leq m.
\]
Then, we have
\[
\Xi\leqslant \Theta\cap \Lambda \leqslant \Lambda\quad \textrm{and}\quad
[\Theta:\Theta\cap \Lambda]\leq [\Gamma:\Lambda]\leq l.
\]
Hence, \eqref{eq:index} implies that
\[
\rk(\Theta\cap \Lambda)\leq [\Theta:\Theta\cap \Lambda] \cdot (\rk(\Theta)-1)+1\leq l(m-1)+1.
\]
Therefore, we obtain the conclusion.
\end{proof}

\begin{proof}[Proof of Proposition~$\ref{prop:wreath_rk}$]
Set $G=H\wr \Gamma$.
For the proof, we may assume that $\genrk(H)$ and $\genrk(\Gamma)$ are finite.
Set $m=\genrk(N)$ and $l=\genrk(\Gamma)$.
Take finitely many elements $g_1,\ldots,g_k\in G$ arbitrarily.
Write $g_i=(v_i,\gamma_i)$ for every $1\leq j\leq k$, where $v_i=\bigoplus_{\Gamma}H$ and $\gamma_i\in \Gamma$.
Since $\genrk(\Gamma)=l$, there exists a subgroup $\Theta$ of $\Gamma$ such that
\[
\langle \gamma_1,\gamma_2,\ldots ,\gamma_k\rangle \leqslant \Theta \leqslant \Gamma \quad \textrm{and}\quad \rk(\Theta)\leq l.
\]
Set $s=\rk(\Theta)$, and take a generating set $\{\theta_1,\ldots ,\theta_s\}$ of $\Theta$ of size $s$.
Set
\[
S=\{v_j(\gamma)\;|\; 1\leq j\leq k,\ \gamma \in \Gamma\};
\]
it is a finite subset of $H$.
Since $\genrk(H)=m$, there exists a subgroup $P$ of $H$ with
\[
\langle S\rangle \leqslant P\leqslant H \quad \textrm{and}\quad \rk(P)\leq m.
\]
Set $\rk(P)=t$ and take a generating set $\{p_1,\ldots ,p_t\}$ of $P$ of size $t$.
Set $f_i=(\mathbf{e},\theta_i) \in \hG$ and $w_n=(p_n\delta_{e_H},e_{\Gamma}) \in \hG$ for every $1\leq i\leq s$ and every $1\leq n\leq t$. Here, $\mathbf{e}$ denotes the map $H\to \Gamma$ sending every $h \in H$ to $e_{\Gamma}$; $p_n\delta_{e_H}$ means the map $H\to \Gamma$ that sends $e_H$ to $p_n$ and sends all the other elements to $e_{\Gamma}$. Then we have
\[
\langle g_1,g_2,\ldots ,g_k\rangle \leqslant \langle f_1,\ldots, f_s, w_1,\ldots ,w_t\rangle \leqslant G
\]
and
\[
\rk(\langle f_1,\ldots, f_s, w_1,\ldots ,w_t\rangle)\leq s+t\leq l+m.
\]
Therefore, we obtain \eqref{eq:loc_intrk_ineq_wreath}, as desired.
\end{proof}

\begin{example} \label{example not equal}
Here we exhibit an example for which the equality does not hold in $\eqref{eq:ind_ineq}$. For every $n \ge 3$,  take an injective homomorphism $f_n\colon F_2\hookrightarrow F_n$ and take an injectve homomorphism $g_n\colon F_n\hookrightarrow F_2$. Then, consider a sequence
\begin{eqnarray}\label{eq:ind_1}
F_2 \xrightarrow{f_3} F_3 \xrightarrow{g_3} F_2 \xrightarrow{f_4} F_4 \xrightarrow{g_4} F_2 \xrightarrow{f_5} F_5 \xrightarrow{g_5} \cdots,
\end{eqnarray}
and let $\Gamma$ be the inductive limit of this sequence. We first claim that
\[
\genrk(\Gamma)=2.
\]
To see this, first we have  $\genrk (\Gamma) \ge 2$ since $\Gamma$ is not locally cyclic. Also,  by applying Proposition~\ref{prop:ind} to the inductive system $\eqref{eq:ind_1}$, we have  $\genrk(\Gamma)\leq \genrk(F_2)=2$. Therefore, we verify the claim.

Now we regard $\Gamma$ as the inductive limit of another inductive system
\begin{eqnarray}\label{eq:ind_2}
\Gamma_1=F_3 \xrightarrow{f_4\circ g_3} \Gamma_2=F_4 \xrightarrow{f_5\circ g_4} \Gamma_3= F_5 \xrightarrow{f_6\circ g_5} \cdots.
\end{eqnarray}
Then, for the inductive system $\eqref{eq:ind_2}$, we have
\[
2=\genrk(\Gamma) < \liminf_{i\to \infty}\genrk(\Gamma_i)=\infty;
\]
in particular, the inequality $\eqref{eq:ind_ineq}$ is \emph{strict} in this setting.
\end{example}

\begin{example} \label{infinite braid}
Recall the definition of the braid group $B_n$ from Example~\ref{ex:braid}.
There exists a natural injective homomorphism from $B_n$ to $B_{n+1}$, and we define $B_\infty$ to be the inductive limit of $B_n$.

In what follows, we show that
\[
\genrk(B_\infty) = 2.
\]
For $n \ge 3$, $B_n$ is non-abelian, and generated by two elements $\sigma_1$ and $\sigma_{n-1} \cdots \sigma_1$ since
\[ (\sigma_{n-1} \cdots \sigma_1)^{-1} \sigma_i (\sigma_{n-1} \cdots \sigma_1) = \sigma_{i+1}\]
for every $1\leq i \leq n-2$. This means that $\genrk(B_n) = 2$. Hence Proposition~\ref{prop:ind} implies that $\genrk(B_\infty) = 2$.

We can also define the inductive limit of the natural inductive system $([B_n,B_n])_{n\geq 2}$; this limit equals $[B_{\infty}, B_{\infty}]$. Then, we have
\[
\genrk([B_{\infty},B_{\infty}])=2.
\]
Indeed, this follows from Proposition~\ref{prop:ind} and  the result \cite{Kordek} of Kordek stating that $\rk([B_n,B_n])=2$ for every $n\geq 7$.
\end{example}

\begin{example}\label{ex:pure}
Here we provide another example related to braid groups; but it has infinite general rank. For every $n\geq 2$, $B_n$ admits a natural surjective homomorphism $B_n\twoheadrightarrow \mathrm{Sym}(n)$, where $\mathrm{Sym}(n)$ denotes the symmetric group of degree $n$. The kernel $P_n$ of this homomorphism is called the \emph{pure braid group with $n$ strands}. We can consider the inductive limit $P_{\infty}$ of the natural inductive system $(P_n)_{n\geq 2}$. Then we have
\[
\genrk(P_{\infty})=\infty.
\]
To see this, we first recall the following well known fact for every $n\geq 2$:
\[
P_{n}^{\ab}\simeq \ZZ^{\binom{n}{2}}
\]
(see \cite[Corollary 1.20]{KTbook}). By construction, we then have
\[
P_{\infty}^{\ab}\simeq \varinjlim \,\ZZ^{\binom{n}{2}}\simeq \bigoplus_{\NN}\ZZ.
\]
Here $(\ZZ^{\binom{n}{2}})_{n\geq 2}$ forms an inductive system via natural inclusion maps. Hence, Proposition~\ref{prop:quot} implies that
\[
\genrk(P_{\infty})\geq \sperk(P_{\infty}^{\ab})=\infty.
\]
\end{example}

In the settings of Examples~\ref{infinite braid} and \ref{ex:pure}, we have a short exact sequence
\begin{eqnarray}\label{eq:braids}
1\longrightarrow P_{\infty}\longrightarrow B_{\infty}\longrightarrow \mathrm{Sym}_{\mathrm{fin}}(\NN)\longrightarrow 1.
\end{eqnarray}
Here, $\mathrm{Sym}_{\mathrm{fin}}(\NN)$ denotes the \emph{finitary symmetric group} on $\NN$: it is the inductive limit of the natural inductive system $(\mathrm{Sym}(n))_{n\geq 2}$. We note that the group $\mathrm{Sym}_{\mathrm{fin}}(\NN)$ is locally finite and that
\[
\genrk(\mathrm{Sym}_{\mathrm{fin}}(\NN))=2.
\]
To see the latter, use Proposition~\ref{prop:quot} (or Proposition~\ref{prop:ind}). From these points of view, $\mathrm{Sym}_{\mathrm{fin}}(\NN)$ might be seen as a `small' group. However, in \eqref{eq:braids} we have
\[
2=\genrk(B_{\infty})<\genrk(P_{\infty})+\genrk(\mathrm{Sym}_{\mathrm{fin}}(\NN))=\infty;
\]
in particular, the inequality~\eqref{eq:loc_intrk_ineq} is \emph{far from being sharp} in this setting.

The following example comes from $1$-dimensional dynamics. Let $F$ be the group of Richard Thompson: it is the group of homeomorphisms on an interval $[0,1]$ that satisfy the following three conditions:
\begin{itemize}
  \item they are piecewise linear;
  \item in the pieces where maps are linear (affine), the slope is in the set $\{2^n\;|\; n\in \ZZ\}$;
  \item the breakpoints are finitely many and they belong to $([0,1]\cap \ZZ[1/2])^2$.
\end{itemize}
It is known that the commutator subgroup $[F,F]$ equals the group consisting of all elements in $F$ that are  identity  in neighborhoods of $0$ and $1$, and that
\[
F^{\ab}=F/[F,F]\simeq \ZZ^2.
\]
It is also known that $[F,F]$ is a simple group with $\rk([F,F])=\infty$, while $\rk(F)=2$. See \cite{CFP} for a comprehensive treatise on $F$.

\begin{prop}\label{prop:F}
For Thompson's group $F$, we have
\[
\genrk([F,F])=2.
\]
\end{prop}

The proof uses standard ideas on $F$ described in \cite[Section~4]{CFP}.

\begin{proof}
Since $[F,F]$ is not locally cyclic, we have $\genrk([F,F])\geq 2$. In what follows, we will show that $\genrk([F,F])\leq 2$. Let $k\in \NN$ and let $f_1,\ldots ,f_k\in [F,F]$. Then, since $f_1,\ldots ,f_k$ are finitely many, there exist $a,b\in \ZZ[1/2]$ with $0<a<1/4<3/4<b<1$ such that for every $1\leq i\leq k$, $f_i$ is identity in $[0,a] \cup [b,1]$. By \cite[Lemma~4.2]{CFP}, there exists $g\in [F,F]$ such that for every $1\leq i\leq k$, $gf_ig^{-1}$ is identity in $[0,1/4]\cup [3/4,1]$. By \cite[Lemma~4.4]{CFP}, the group $H$ consisting of all elements in $[F,F]$ that are identity in $[0,1/4]\cup [3/4,1]$ is isomorphic to $F$; hence it is generated by two elements. Therefore, we have
\[
\langle f_1,\ldots ,f_k\rangle \leqslant g^{-1}Hg\leqslant [F,F]
\]
with $\rk(g^{-1}Hg)=2$. This yields
\[
\genrk([F,F])\leq 2,
\]
as desired.
\end{proof}

With the aid of Propositions~\ref{prop:ind}, \ref{prop:ext_rk} and \ref{prop:wreath_rk}, we can build up groups of finite general rank  out of groups that are known to have finite general rank. For instance, the group $\QQ\wr (B_{\infty}\wr [F,F])$ is a group of general rank at most $5$ by Example~\ref{infinite braid} and Proposition~\ref{prop:F}.

In relation to Theorem~\ref{thm:wreath_shin}, the following problem seems interesting.
\begin{problem}
Given a group $\Gamma$ $($non-abelian and not finitely generated$)$, determine $\genrk(\Gamma)$.
\end{problem}

\section*{Acknowledgment}
The authors are grateful to the anonymous referee for useful comments, which improve the present paper.
The second author is supported by JSPS KAKENHI Grant Number JP20H00114 and JST-Mirai Program Grant Number JPMJMI22G1.
The third author is supported by JSPS KAKENHI Grant Number JP21J11199.
The first author, the fourth author and the fifth author are partially supported by JSPS KAKENHI Grant Number JP21K13790, JP19K14536 and JP21K03241, respectively.


\end{document}